\documentclass{amsart}
\usepackage{amsmath}
\usepackage{amsfonts}
\usepackage{amssymb}
\usepackage[T1]{fontenc}
\usepackage{xcolor}
\usepackage{cleveref}
\crefformat{section}{\S#2#1#3} 
\crefformat{subsection}{\S#2#1#3}
\crefformat{subsubsection}{\S#2#1#3}

\newtheorem{theorem}{Theorem}[section]
\newtheorem{corollary}{Corollary}[theorem]
\newtheorem{lemma}[theorem]{Lemma}
\newtheorem{definition}{Definition}[section]
\theoremstyle{remark}
\newtheorem{remark}[corollary]{Remark}

\newcommand{\R}{\mathbb{R}}

\newcommand{\Rd}{\mathbb{R}^d}

\newcommand{\Z}{\mathbb{Z}}

\newcommand{\T}{\mathbb{T}}
\newcommand{\TT}{\mathbb{T}^2}

\newcommand{\C}{\mathbb{C}}

\newcommand{\N}{\mathbb{N}}

\DeclareMathOperator\supp{supp}
\newcommand\norm[1]{\left\lVert#1\right\rVert}
\newcommand\abs[1]{\left\lvert#1\right\rvert}

\usepackage{cite}
\title{Global Well-posedness for the periodic fractional cubic NLS in 1D}
\author{Alexandre Megretski}
\address{LIDS, Massachusetts Institute of Technology, Cambridge, MA 02139-4307}
\email{ameg@mit.edu}

\author{Nikolaos Skouloudis} 
\address{LIDS, Massachusetts Institute of Technology, Cambridge, MA 02139-4307}
\email{nskoulou@mit.edu}

\begin{document}
\begin{abstract}
We consider the defocusing periodic fractional nonlinear Schrödinger equation
\begin{equation*}
     i \partial_t u +\left(-\Delta\right)^{\alpha}u=-\abs{u}^2u,
\end{equation*}
where $\frac{1}{2}< \alpha < 1$ and the operator $(-\Delta)^\alpha$ is the fractional Laplacian with symbol $\abs{k}^{2\alpha}$. We establish global well-posedness in $H^s(\T)$ for $s\geq \frac{1-\alpha}{2}$ and we conjecture this threshold to be sharp as it corresponds to the pseudo-Galilean symmetry exponent. Our proof uses the $I$-method to control the $H^s(\T)$-norm of solutions with infinite energy initial data. A key component of our approach is a set of improved long-time bilinear Strichartz estimates on the rescaled torus, which allow us to exploit the subcritical nature of the equation.
\end{abstract}
\maketitle
\section{Introduction}
  We consider the defocusing periodic fractional nonlinear Schrödinger equation
\begin{equation}
    \label{eq:fractional_main_lwp}
    \begin{cases}
       i \partial_tu +\left(-\Delta\right)^{\alpha}u=-\abs{u}^2u, \\
       u(x,0)=u_0(x) \in H^{s}\left(\T\right),
    \end{cases}
    \tag{fNLS}
\end{equation}
where $\T=\R/\left(2\pi \Z \right)$ is the unit circle, $\frac{1}{2}< \alpha < 1$ and the operator $(-\Delta)^\alpha$ is the fractional Laplacian with symbol $\abs{k}^{2\alpha}$. The fractional nonlinear Schrödinger equation was originally introduced by Laskin \cite{Lakshin} on the Euclidean space ($\Rd$) in the study of fractional quantum mechanics. Moreover, on $\mathbb{R}$, (\ref{eq:fractional_main_lwp}) was formally derived in \cite{Kirkpatrick_Staffilani} as the continuum limit of a discrete fractional nonlinear Schrödinger equation with long-range lattice interactions. See also \cite{Hong_strong_convergence,Hong_strong_convergence_periodic} for similar results in higher dimensions and periodic domains. The aim of this paper is to prove global well-posedness of the Cauchy problem \eqref{eq:fractional_main_lwp} in $H^s(\T)$ for $s \geq \frac{1-\alpha}{2}$.\\
We recall that the initial value problem \eqref{eq:fractional_main_lwp} is \textit{locally well-posed} if, for any $u_0 \in H^s(\T) $, there exists a time $T>0$, an open ball $B$ in $H^s(\T)$ containing $u_0$, and a closed subspace, $X$, of $C^0([0,T] \rightarrow H^s(\T))$, such that for each $u_0 \in B$ there exists a unique $u \in X$ satisfying,
\begin{equation*}
    u(t)=S(t)u_0+i \int_0^t S(t-t') \abs{u(t')}^2u(t') dt',
\end{equation*}
where $S(t)$ is defined to be fractional Schrödinger group (see \cref{sec:2}). Moreover, we require the map $u_0 \mapsto u$ to be uniformly continuous as a map from $B$ (endowed with the $H^s(\T)$ topology) to $X$ (endowed with the $C^0([0,T] \rightarrow H^s(\T))$ topology). If we can take $T>0$ to be arbitrarily large we say that (\ref{eq:fractional_main_lwp}) is \textit{globally well-posed}. \\
When $\alpha=1$, \eqref{eq:fractional_main_lwp} reduces to the classical cubic nonlinear Schrödinger equation (NLS) which is completely integrable. The local theory was established by Bourgain \cite{Bourgain_lwp}, who employed number theoretic techniques to obtain sharp Strichartz estimates, proving local well-posedness for $s \geq 0$. For fractional $\frac{1}{2}<\alpha<1$, the $L_{t,x}^4(\TT)$ Strichartz estimates were generalized in \cite{demirbas_fractional}, leading to local well-posedness for $s>\frac{1-\alpha}{2}$. The endpoint case, $s=\frac{1-\alpha}{2}$, was subsequently addressed in \cite{CHO_2014} by adapting the multilinear estimate techniques developed in \cite{Tao_dyadic}. Although the focus of this paper is on the periodic setting, it is worth mentioning that in \cite{Sire_Hong_fractional_lwp_euclidean}, local well-posedness was established on the real line for the same range of regularity. We also refer the reader to \cite{Ionescu_Pusateri,dinh} for results concerning $0<\alpha<\frac{1}{2}$ and more general power-type nonlinearities. \\
The aforementioned local well-posedness results are expected to be sharp as they correspond to pseudo-Galilean invariance exponent $s_g:=\frac{1-\alpha}{2}$ (see \cite{Sire_Hong_fractional_lwp_euclidean}). We recall that the pseudo-Galilean invariance exponent corresponds to the transformation,
\begin{equation*}
    u^v(x,t):= e^{i\left(v x-\abs{v}^{2\alpha}t\right)}u\left(x- 2\alpha\abs{v}^{2(\alpha-1)}vt,t\right), \text{   } u_0^v(x):= e^{iv x}u_0\left(x\right),
\end{equation*}
and leaves (\ref{eq:fractional_main_lwp}) invariant only when $\alpha=1$. In particular, for $\alpha=1$ and $s<s_g=0$, one can use phase decoherence to prove that the data-to-solution map is not uniformly continuous even on an arbitrarily short time intervals \cite{Burq_illposedness_T,Tao_disperive_book}, hence demonstrating the sharpness of Bourgain's result. An analogous construction to \cite{Burq_illposedness_T} yields ill-posedness for $s<0$ and $\frac{1}{2}<\alpha<1$. Moreover, on the real line, by adapting the techniques used in \cite{Christ_Tao_illposedness_2}, Cho et al.~\cite{CHO_2014} exploited the ill-posedness of the classical NLS together with the pseudo-Galilean transformation to prove ill-posedness when $s<s_g$. As observed by Cho et al.~\cite{CHO_2014}, it is natural to conjecture that ill-posedness on $\T$ arises from initial data whose Fourier support is contained in an interval of length $N^{1-\alpha}$ centered at some large frequency $N$. Indeed, the aforementioned example demonstrates that the trilinear estimates used in the proof of local well-posedness (see Lemma \ref{lemma:trilinear_estimates_lwp}) can only hold for $s\geq s_g.$ Although we were unable to leverage this intuition to show failure of local uniform continuity of the data-to-solution map, we managed to prove that the nonlinear part of the first Picard iterate is unbounded in $H^s(\T)$ for $s<s_g$. This implies that the data-to-solution map cannot be $C^3$ — see Theorem \ref{theorem:main_illposedness_theorem_n=1} for a precise statement. \\
It is also useful to recall the scaling critical exponent, $s_c=\frac{1}{2}-\alpha$, which corresponds to the transformation
\begin{equation*}
    u^\lambda(x,t):= \frac{1}{\lambda^\alpha} u\left(\frac{x}{\lambda}, \frac{t}{\lambda^{2\alpha}}\right), \text{   } u_0^\lambda(x):= \frac{1}{\lambda^\alpha} u_0\left(\frac{x}{\lambda}\right),
\end{equation*}
and leaves (\ref{eq:fractional_main_lwp}) invariant but now $u^\lambda$ is a function on $\T_\lambda \times \R$ where $\T_\lambda := \R /(2\pi\lambda\Z)$. We emphasize that since $s_c<0<s_g$, the local well-posedness regularity threshold is dictated by the pseudo-Galilean symmetry. In fact, \eqref{eq:fractional_main_lwp} is mass-subcritical, a crucial property that will be exploited in our proof (see Lemma \ref{lemma: lwp_ILNS}). \\
Furthermore, by recalling that smooth solutions to (\ref{eq:fractional_main_lwp}), conserve the mass and energy,
\begin{subequations}
\begin{align}
 \label{def: mass} 
    M\left(u(t)\right) &= \norm{u(t)}_{L^2(\T)},\tag{mass} \\
 \label{def: energy}
    E(u(t)) &= \frac{1}{2}\norm{\left(-\Delta\right)^\frac{\alpha}{2} u(t)}_{L^2(\T)}^2 +\frac{1}{4} \norm{u(t)}^4_{L^4(\T)},   \tag{energy}  
\end{align}
\end{subequations}
we can use the persistence of regularity and Bourgain's local well-posedness result \cite{Bourgain_lwp} to obtain global well-posedness for $s\geq 0$ when $\alpha=1$. In the fractional case, global well-posedness similarly follows from the local theory \cite{demirbas_fractional} when $s \geq \alpha$. By employing the high-low method introduced in \cite{Bourgain_HighLow}, Demirbas et al.~\cite{demirbas_fractional} further extended global well-posedness to the regime $s>\frac{10\alpha+1}{12}$. The goal of this paper is improve this threshold by using the $I$-method. Originally introduced in \cite{I-method_1,I-method_2}, the $I$-method has become a canonical tool in obtaining global well-posedness for infinite energy initial data. The method relies on constructing families of modified Hamiltonians which are almost conserved. This approach has been successfully applied to a range of dispersive equations, including the cubic and quintic nonlinear Schrödinger equations \cite{mass_critical_gwp_staffilani_tzirakis, Schippa_mass_critical, mass_critical_Li}, the derivative NLS \cite{I-method_2}, and the KdV equation \cite{KdV}, both in periodic and Euclidean settings.\\
In our setting, we establish global well-posedness for $s\geq\frac{1-\alpha}{2}$. The main ingredients of the proof are linear $L^6_{t,x}$ Strichartz estimates (see Lemma \ref{lemma_Strichartz_Linear}), which follow immediately from the $\ell^2$ decoupling inequality \cite{Bourgain_Demeter_decoupling,Schippa_fractional}, and improved long-time bilinear Strichartz estimates on the rescaled torus $\T_\lambda=\R/\left(2\pi\lambda \Z\right)$ (see Lemma \ref{lemma_Strichartz_bilinear}). The latter estimates are then used to quantify the growth of the modified Hamiltonian, which in turn yields polynomial-in-time bounds for the $H^s(\T)$ norm of solutions to (fNLS). This immediately implies existence and uniqueness of solutions for arbitrarily large time intervals. More formally, we prove the following results.
\begin{theorem}
\label{theorem:main_gwp_theorem_n=1}
(\ref{eq:fractional_main_lwp}) is globally well-posed for $s\geq \frac{1-\alpha}{2}$. 
\end{theorem}
\begin{theorem}
\label{theorem:main_illposedness_theorem_n=1}
Suppose that $s< \frac{1-\alpha}{2}$ and let $r>0$ and $T >0$ be arbitrarily small constants. Assume that the data-to-solution map $u_0 \mapsto u(\cdot)$ associated with (\ref{eq:fractional_main_lwp}) for smooth initial data extends continuously to a map from the closed ball in $H^s(\T)$ of radius $r$ centered at the origin into $C^0\left([0,T] \rightarrow H^s(\T)\right)$. Then, this map is not $C^3$ at the origin.
\end{theorem}
 Our paper is organized as follows: in \cref{sec:2}, we introduce some notation and recall some basic definitions, in \cref{sec:3} we prove linear and bilinear Strichartz estimates, in \cref{sec:4} we give the proof of Theorem \ref{theorem:main_gwp_theorem_n=1} using the $I$-method and in \cref{sec:5} we prove Theorem \ref{theorem:main_illposedness_theorem_n=1}.
\section{Preliminaries and Notation} \label{sec:2}
 We  use $A \lesssim B$ to indicate an estimate of the form $A \leq CB$ for some constant $C>0$. When we want to stress that the constant depends on a parameter, say $p$, we write  $A \lesssim_p B$. If  both $A \lesssim B$ and $B \lesssim A$ hold, we write $A \sim B$. Furthermore, $A \ll B$ signifies that $A \leq cB$ for some small constant $0<c<1$. We also use the shorthand notation $a\pm$ to denote $a \pm \epsilon$ for some $0<\epsilon \ll 1$. This notation is employed only in proofs, and in the formal statements of theorems we always write the precise exponent. Finally, $\mathbb{N}$ denotes the set of natural numbers and we write $\mathbb{N}_0:=\mathbb{N} \cup \{0\}$ for the set of non-negative integers. \\ 
Define $\T_\lambda:= \R/\left(2\pi \lambda \Z\right)$ where $\lambda\geq 1$ will be assumed throughout the remainder of this paper. For $f: \T_\lambda \rightarrow \mathbb{C}$ measurable we define Lebesgue norms by
\begin{equation*}
    \norm{f}_{L^p(\T_\lambda)}^p := \int_{\T_\lambda} \abs{f(x)}^p \,\, dx,
\end{equation*}
for $1 \leq p < \infty$ and the usual modification for $p=\infty$. \\
We next define $\left(dk\right)_\lambda$ to be the normalized counting measure on $\Z_\lambda := \frac{1}{\lambda}\Z$
\begin{equation*}
    \int_{\Z_\lambda} a(k) \left(dk\right)_\lambda := \frac{1}{2\pi\lambda} \sum _{k \in \Z_\lambda} a(k).
\end{equation*}
The Fourier coefficients of $f \in L^1(\T_\lambda)$ are given by
\begin{equation*}
    \hat{f}(k) = \int_{\T_\lambda} e^{-i k x} f(x) \, \, dx,
\end{equation*}
for $k \in \Z_\lambda$ and the Fourier inversion formula is given by
\begin{equation*}
    f(x)= \int_{\Z_\lambda} e^{i k x} \hat{f}(k)\left(dk\right)_\lambda.
\end{equation*}
Using this convention, the following identities hold
\begin{subequations}
\begin{align*}
   & \int_{\T_\lambda} f(x) \overline{g(x)} \,\, dx = \int_{\Z_\lambda} \hat{f}(k) \overline{\hat{g}(k)} \,\, (dk)_\lambda,  \\ 
&\norm{f}_{L^2(\T_\lambda)}=\lVert \hat{f}\rVert _{L^2((dk)_\lambda)}, \\ 
   &\widehat{fg}(k) = \int_{\Z_\lambda} \hat{f}(k_1) \hat{g}(k-k_1) \,\, (dk_1)_\lambda.  
\end{align*}
\end{subequations}
The Sobolev space, $H^s=H^s(\T_\lambda)$, is defined as the completion of smooth functions under the norm
\begin{equation*}
    \norm{f}_{H^s}=\norm{\langle k \rangle^s \hat{f}(k)}_{L^2((dk)_\lambda)},
\end{equation*}
where $\langle k \rangle := (1+\abs{k}^2)^\frac{1}{2}$. We will frequently make use of Littlewood-Paley theory which allows us to quantitatively separate the rough, high-frequency behavior of a function from its smooth, low-frequency part. In particular, we denote by $P_N$ the Fourier projection operator  
\begin{equation*}
    \widehat{P_N f}:= \begin{cases}
        \hat{f}(k) &\text{ for } \frac{N}{2} \leq \abs{k} < N,\\
        0         &\text{ otherwise }
    \end{cases}
\end{equation*}
when $N>1$ and $P_1$ is defined to be the projection onto frequencies $[-1,1]$.
 We define $S(t)$ to be the solution operator to the fractional linear Schrödinger equation
\begin{equation*}
    i \partial_t u + (-\Delta)^\alpha u =0, \text{ } u(x,0)=u_0(x), \text{  } x \in \T_\lambda,
\end{equation*}
that is,
\begin{equation*}
    S(t) u_0(x) = \int_{\Z_\lambda} e^{ i(k x+\abs{k}^{2\alpha}t)} \hat{u}_0(k) (dk)_\lambda.
\end{equation*}
\subsection{The $I$-operator} \label{sec: I-operator}
We recall the main definitions and conventions used in the construction of the $I$-operator. Define $m_1$ as the restriction to $\Z_\lambda$ of the smooth and monotone multiplier satisfying
\begin{equation*}
    m_1(x):=\begin{cases}
        1 &\text{ if } \abs{x} \leq 1, \\
        \abs{x}^{-1} &\text{ if } \abs{x} > 2,
    \end{cases}
\end{equation*}
and $\sup_{x \in \R} \abs{x^2m_1''(x)} \lesssim 1$. For $N>0$, set $m_N(x):=m_1(\frac{x}{N})$ and for $\beta\geq 0$, we let $I_N^\beta$ denote the Fourier multiplier satisfying $\widehat{I_N^\beta f}:=m^\beta_N\hat{f}$. Additionally, note that $I_N^\beta$ is a smoothing operator of degree $\beta$
\begin{equation*}
    \norm{u}_{H^{s_0}(\T_\lambda)} \lesssim_{\beta} \norm{I_N^\beta u}_{H^{s_0+\beta}(\T_\lambda)}\lesssim_{\beta} N^{\beta}  \norm{u}_{H^{s_0}(\T_\lambda)},  \text{   } \forall s_0 \in \R.
\end{equation*}
We also recall an interpolation lemma which is useful for proving local well-posedness. To this end, for every $x \in \T_\lambda$ define $T_x$ to be the translation operator $$T_x u(x',t):=u(x'+x,t).$$ A Banach space $X$, of spacetime functions defined on $\T_\lambda \times \R$, is translation invariant if $\norm{T_x u}_X = \norm{u}_X$ for all $u \in X$ and all $x \in \T_\lambda$. Similarly, an $n$-linear operator $F$ is said to be translation invariant if
\begin{equation*}
    F(T_x u_1, \dots, T_x u_n) = F(u_1, \dots, u_n),
\end{equation*}
for all $x \in \T_\lambda$ and all $u_1, \dots,  u_n \in X$.
\begin{lemma} \cite[Lemma 12.1]{Invariant_Lemma}
\label{lemma: I-method_invariant_lemma}
    Let $\beta>0$, $n\geq \N$ and suppose that $Z, X_1, \dots X_n$ are translation invariant Banach spaces and $F$ is a translation invariant $n$-linear operator such that the following estimate holds
    \begin{equation*}
        \norm{I_1^\beta F(u_1,\dots,u_n)}_Z \lesssim \prod_{j=1}^n \norm{I_1^\beta u_j }_{X_j},
    \end{equation*}
    for all $0 \leq \beta \leq \beta_0$ and $u_j \in X_j$. Then one has the estimate,
        \begin{equation*}
        \norm{I_N^\beta F(u_1,\dots,u_n)}_Z \lesssim \prod_{j=1}^n \norm{I_N^\beta u_j }_{X_j},
    \end{equation*}
    for all $0 \leq \beta \leq \beta_0$,  $u_j \in X_j$ and $N \geq 1$, with explicit constant independent of $N$.
\end{lemma}
Throughout the remainder of the paper, we write $m:=m_N^{\alpha-s}$ and $I:=I_N^{\alpha-s}$.
\subsection{Function spaces}
As we will be working with long-time Strichartz estimates, it is convenient to prove local well-posedness results in modified $X^{s,b}$ spaces. One option would be the adapted functions spaces $U^p$ and $V^p$ which were developed for the critical regularity well-posedness theory of dispersive equations on periodic domains \cite{Herr_Tartaru_Tzvetkov,Hadac_Herr_Koch,Koch_Tataru} and were more recently used to prove global well-posedness for the mass critical NLS on $\T$ \cite{Schippa_mass_critical} and $\TT$ \cite{Herr_Kwak_2024_mass_critical} in the sub-critical setting. However, when using the $I$-method, these spaces require non-trivial resonant analysis as they are not $L^2$ based. Instead, we consider a rescaled Bourgain space adapted to the interval $[0,T]$. 
\begin{definition}
For any $b \in \mathbb{R}$ and $T>0$, we define the rescaled Bourgain space $X^{s,b}_T$, as the completion of Schwarz functions $f: \R \times \T_\lambda \rightarrow \mathbb{C}$ under the norm
\begin{equation*}
   \norm{u}_{X^{s,b}_T}^2:=T^{-1} \int_\R \int_{\Z_\lambda} \langle \omega-\abs{k}^{2\alpha} \rangle_T^{2b} \langle k \rangle^{2s} \abs{\tilde{u}(\omega,k)}^2 \left(dk\right)_\lambda d\omega,
\end{equation*}
where $\langle \omega \rangle_T:=\left(1+T^2\omega^2\right)^\frac{1}{2}$ and $\tilde{u}(\omega,k)$ denotes space-time Fourier transform
\begin{equation*}
    \tilde{u}(\omega,k):=\int_\R \int_{\T_\lambda} u(t,x) e^{i(\omega t+kx)} dx dt.
\end{equation*} 
For a time interval $J \subset \R$ of length $\abs{J}=T$, we define the time-restricted norm
\begin{equation*}
    \norm{f}_{X^{s,b}_T(J)}:= \inf \left\{\norm{g}_{X^{s,b}_T} : f\equiv g \text{ on } J \times \T_\lambda \right\}.
\end{equation*}
\end{definition}
  We remark that the space $X^{s,b}_T$ is defined so that it is invariant under the corresponding time rescaling: if we set $g(t,x):=f(t/T, x)$, then $\norm{g}_{X^{s,b}_T}=\norm{f}_{X^{s,b}_T}$. This property ensures that smoothed-out indicator functions adapted to intervals of length $T$ have $X^{s,b}_T$ norm $\sim 1$ (see Lemma \ref{lemma: X(s,b) properties}(a)). \\
We also record the some important properties of the $X^{s,b}_T$ spaces which will be useful for converting Strichartz estimates to $L^p$ norm estimates (see \cref{sec:3}) and for proving local well-posedness theorems (see \cref{sec:4.1}). To this end, it is useful to note the following equivalent characterization of the $X^{s,b}_T$ space
\begin{equation*}
    \norm{f}_{X^{s,b}_T}=\norm{S(-\cdot) f(\cdot) }_{H^b_T(\R \rightarrow H^s\left(\T_\lambda)\right)},
\end{equation*}
where $S(-\cdot) f(\cdot)$ denotes the function $t \mapsto S(-t) f(t)$ and $H^b_T$ is the modified Sobolev norm in which the usual weight $\langle \omega \rangle^b $ is replaced with $T^{-\frac{1}{2}}\langle \omega \rangle_T^b $.
\begin{lemma}
\label{lemma: X(s,b) properties} The $X^{s,b}_T$ spaces satisfy the following properties:
\begin{itemize}
    \item[(a)] For any $b \in \R$, there exists a constant $C=C(b)>0$ such that for every $s\in \R$, $T>0$ and $\phi \in H^s$, we have the estimate
        \begin{equation*}
            \norm{S(t)\phi}_{X^{s,b}_T\left([0,T]\right)} \leq C \norm{\phi}_{H^s}.
        \end{equation*} 
    \item[(b)] For any $b, b' \in \R$ satisfying $-\frac{1}{2}<b'$ and $ b \leq 1+b'$, there exists a constant $C=C(b,b')>0$ such that for every $s\in \R$, $T>0$ and $f \in X^{s,b'}$, we have the estimate
        \begin{equation*}
            \norm{\int_0^t S(t-t') f(t') \,\,dt'}_{X^{s,b}_T\left([0,T]\right)}  \leq CT \norm{f}_{X^{s,b'}_T}.
        \end{equation*}
\end{itemize}
\begin{proof}
    For (a), it suffices to observe that a fixed smooth, compactly supported cut-off function $\eta_J : \R \to \R$ adapted to the interval $J$ satisfies $\norm{\eta_J}_{H^b_T} \sim 1$. \\
    For part (b), it is clear that suffices to prove
    \begin{equation*}
         \norm{\int_0^t f(t') \,\,dt'}_{H^b_T\left([0,T]\right)}  \leq CT \norm{f}_{H^{b'}_T},
    \end{equation*}
    for all $f \in H^{b'}_T(\R \rightarrow \mathbb{C})$. The proof of this estimate is analogous to \cite[Lemma 3.2]{GINIBRE} and is therefore omitted.
\end{proof}
\end{lemma}
\begin{lemma} 
\label{lemma: transference multilinear X(s,b)}
Let $n \in \Z$, $s\in \R$ and suppose that $Y$ is a Banach space of functions on $\R \times \T_\lambda$ with the property that there exists a constant $C_1>0$ such that
\begin{equation}
    \norm{\prod_{j=1}^n e^{i\omega_j t} S(t) f_j}_Y\leq C_1 \prod_{j=1}^n \norm{f_j}_{H^s},
    \label{hypothesis: Y}
\end{equation}
for all $f_1, \dots, f_n \in H^s(\T_\lambda)$ and $\omega_1, \dots , \omega_n \in \R $. Then for any $b>\frac{1}{2}$, there exists a constant $C_2=C_2(b)>0$ such that for any $T>0$ and $u_1, \dots , u_n \in X^{s,b}_T$ we have the following estimate
\begin{equation*}
    \norm{\prod_{j=1}^n u_j}_Y \leq C_1 C_2\prod_{j=1}^n \norm{u}_{X^{s,b}_T},
\end{equation*}
where the constant $C_1$ is the same as the one in \eqref{hypothesis: Y}.
\begin{proof}
The proof is analogous to \cite[Chapter 2, Lemma 2.9]{Tao_disperive_book}. By Fourier inversion, we have
\begin{equation*}
    u_j(x,t)=\frac{1}{2\pi} \int_\R \int_{\Z_\lambda} \widetilde{u}_j(\omega,k) e^{i(\omega t+kx)} \left(dk_\lambda\right) d\omega,
\end{equation*}
for $j=1, \dots, n$. Writing $\omega=\omega_j+\abs{k}^{2\alpha}$ and setting
\begin{equation*}
    f_{\omega_j}(x) := \int_{\Z_\lambda} \widetilde{u}_j(\omega_j+\abs{k}^{2\alpha},k) e^{ikx}\left(dk_\lambda\right),
\end{equation*}
we have the representation
\begin{equation*}
    u_j(t)=\frac{1}{2\pi} \int_\R e^{i\omega_j t}S(t) f_{\omega_j} d\omega_j.
\end{equation*}
By \eqref{hypothesis: Y} and Minkowski's inequality, followed by Cauchy-Schwarz, we obtain
\begin{align*}
    \norm{\prod_{j=1}^n u_j}_Y &\lesssim C_1 \prod_{j=1}^n \int_\R \norm{f_{\omega_j}}_{H^s} d\omega_j, \\
    &\leq  C_1  \prod_{j=1}^n  \left(\int_\R T \langle \omega_j \rangle_T^{-2b} d\omega_j\right)^{\frac{1}{2}} \norm{u_j}_{X^{s,b}_T},  \\
     &\lesssim  C_1 \prod_{j=1}^n \norm{u_j}_{X^{s,b}_T}. \qedhere
\end{align*}
\end{proof}
\end{lemma}
As a corollary of Lemma \ref{lemma: transference multilinear X(s,b)} with $n=1$ and $Y = C^0([0,T]; H^s)$, we obtain the continuous embedding $X^{s,b}_T \hookrightarrow C^0([0,T]; H^s)$ for every $T>0$.
\begin{corollary}
\label{corollary: C^0 embedding}
   For any $b>\frac{1}{2}$, there exists a constant $C=C(b)>0$ such that for every $s\in \R$, $T>0$ and $u \in X^{s,b}_T\left([0,T]\right)$, the following estimate holds
\begin{equation*}
   \sup_{t \in [0,T] } \norm{u(t)}_{H^s} \leq C \norm{u}_{X^{s,b}_T\left([0,T]\right)}.
\end{equation*}
\end{corollary} 
\section{Strichartz Estimates}\label{sec:3}
 Throughout this section, all frequency scales $N$ appearing in the projections $P_N$ are assumed to satisfy $N \geq 1$, unless stated otherwise. We also recall that $\frac{1}{2}<\alpha<1$ and $\lambda \geq 1$. \\
 In light of the rescaling symmetry inherent to the problem, it is natural to establish Strichartz estimates on the rescaled torus $\T_\lambda$. To facilitate this, we begin by stating a rescaling lemma that enables us to transfer the analysis back to the standard torus $\T$.
\begin{lemma}
\label{Lemma: Conversion_between_T_and_T_lambda}
    Let $T>0$, $n\in \N$ and $N_1, \dots N_n >0$. Then the following two statements are equivalent:
    \begin{enumerate}
        \item[(i)] There exists a constant $C_1(T, N_1, \dots , N_n)>0$ such that for all $f_j \in L^2(\T)$ satisfying $\supp \hat{f}_j \subset \left\{k \in \Z:  \frac{N_j}{2} \leq \abs{k} < N_j \right\}$ for $1\leq j \leq n$, we have the following estimate    
    \begin{equation*}
        \int_{[0,T] \times \T} \prod_{j=1}^n \abs{S(t)f_j}^2 \leq C_1(T, N_1, \dots , N_n) \prod_{j=1}^n \norm{f_j}_{2}^2.
    \end{equation*}
    \item[(ii)] There exists a constant $C_\lambda(T, N_1, \dots , N_n)>0$ such that for all $\phi_j \in L^2(\T_\lambda)$ satisfying $\supp \hat{\phi}_j \subset \left\{k \in \Z_\lambda: \frac{N_j}{2} \leq \abs{k} < N_j \right\}$ for $1\leq j \leq n$, we have the following estimate
     \begin{equation*}
        \int_{[0,T] \times \T_\lambda} \prod_{j=1}^n \abs{S(t)\phi_j}^2 \leq C_\lambda(T, N_1, \dots , N_n)\prod_{j=1}^n \norm{\phi_j}_{2}^2.
\end{equation*}
\end{enumerate}
Moreover, the constants are related by
\begin{equation*}
    C_\lambda(T, N_1, \dots , N_n)=\lambda^{1-n+2\alpha}C_1(\lambda^{-2\alpha}T, \lambda N_1, \dots , \lambda N_n).
\end{equation*}
\begin{proof}
We will only prove that (i) $\implies$ (ii) as the reverse direction is trivial. Suppose that $\phi_j \in L^2(\T_\lambda)$ are such that $\supp \hat{\phi}_j \subset \left\{k \in \Z_\lambda: \frac{N_j}{2} \leq \abs{k} < N_j\right\}$ for $1\leq j \leq n$. For fixed $1\leq j \leq n$, define $f_j(x):=\phi_j(\lambda x)$ and note that ${\hat{f}_j(k)=\lambda^{-1} \hat{\phi}_j(\lambda^{-1} k)}$ for all $k \in \Z$ and $\norm{f_j}_{L^2(\T)}^2 = \lambda^{-1} \norm{\phi_j}_{L^2(\T_\lambda)}^2$. Therefore,
    \begin{align*}
        \int_{[0,T] \times \T_\lambda} \prod_{j=1}^n \abs{S(t) \phi_j}^2 &= \lambda^{1+2\alpha} \int_{[0,\lambda^{-2\alpha}T] \times \T} \prod_{j=1}^n \abs{S(t)f_j}^2, \\
        &\leq \lambda^{1+2\alpha} C(\lambda^{-2\alpha}T, \lambda N_1, \dots , \lambda N_n) \prod_{j=1}^n \norm{f_j}_{L^2(\T)}^2, \\
        & =  \lambda^{1-n+2\alpha} C(\lambda^{-2\alpha}T, \lambda N_1, \dots , \lambda N_n) \prod_{j=1}^n \norm{\phi_j}_{L^2(\T_\lambda)}^2,
    \end{align*}
where the first line follows by a coordinate change, the second line follows from statement (i) and the last line follows by changing coordinates again.
\end{proof}
\end{lemma}
We proceed by establishing sharp bilinear Strichartz estimates via an elementary counting argument. The latter yields a sharp $L_{t,x}^4$ estimate, which represents an improvement over the estimate proved in \cite{demirbas_fractional} which incurred an extra $\epsilon$ derivative loss. Similar bilinear estimates can be found in \cite{Schippa_fractional} but these are only valid for short time intervals that are frequency dependent. Instead, we prove an inequality that is valid for all time and generalizes the bilinear estimate proved for $\alpha=1$ in \cite{mass_critical_gwp_staffilani_tzirakis}. In this respect, we first state an elementary counting lemma which will be used repeatedly in the remainder of this section.
\begin{lemma}
    \label{Lemma: elementary counting lemma}
   Let $a,b \in \Z$ with $a<b$ and let $B \subset \R$ be an interval of bounded length. For $f: \R \rightarrow \R$, define $M_f:= \# \left\{ k \in [a,b]\cap \Z : f(k) \in B \right\}$. Then, there exists a constant $C>0$ such that:
   \begin{enumerate}
        \item[(i)] If $f \in C^1\left([a,b]\right)$ and $\min_{x \in [a,b]} \abs{f'(x)} \geq r>0$, then $$M_f \leq C\left( 1+r^{-1}\right).$$
        \item[(ii)] If $f \in C^2\left([a,b]\right)$ and $\min_{x \in [a,b]}f''(x) \geq r>0$, then $$M_f \leq C \left(1+r^{-\frac{1}{2}}\right).$$
    \end{enumerate}
\begin{proof}
The estimates in (i) and (ii) are straightforward applications of the mean value theorem - see for example \cite{Graham_Kolesnik_1991}.
\end{proof}
\end{lemma}
\begin{lemma}
    \label{lemma_Strichartz_bilinear}
    Let $\phi_1, \phi_2 \in L^2(\T_\lambda)$ be such that $\phi_1  = P_{N_1} \phi_1$  and $\phi_2=P_{N_2} \phi_2$. Then there exists a constant $C=C(\alpha)>0$ such that for any $T>0$, the following estimate holds
\begin{equation*}
\norm{S(t)\phi_1 S(t)\phi_2}_{L_{t,x}^2([0,T] \times \T_\lambda)} \leq C\left(\frac{T}{\lambda}+T^\frac{1}{2}\min\left(N_1,N_2\right)^{1-\alpha}\right)^{\frac{1}{2}} \norm{\phi_1}_2\norm{\phi_2}_2.
\end{equation*}
Moreover, if $ \max\left(N_1,N_2\right)\geq 4\min\left(N_1,N_2\right)$, the following refined estimate holds
    \begin{equation*}
        \norm{S(t)\phi_1 S(t)\phi_2}_{L_{t,x}^2([0,T] \times \T_\lambda)} \leq C\left(\frac{T}{\lambda}+\max\left(N_1,N_2\right)^{1-2\alpha}\right)^{\frac{1}{2}} \norm{\phi_1}_2 \norm{\phi_2}_2.
    \end{equation*}
\begin{proof}
By Lemma \ref{Lemma: Conversion_between_T_and_T_lambda} it suffices to prove the statement for $\lambda=1$. Next note that,
\begin{equation*}
     \norm{S(t)\phi_1 S(t)\phi_2}_{L_{x}^2(\T)}^2 = \sum_{k,k_1,k_2} \widehat{\phi}_1(k_1)\overline{\widehat{\phi}_1(k_2)}\widehat{\phi}_2(k-k_1)\overline{\widehat{\phi}_2(k-k_2)} e^{it\psi(k,k_1,k_2)}
\end{equation*}
where the summation is restricted to 
\begin{equation*}
    N_1 \leq \abs{k_1}, \abs{k_2}\leq2N_1 \text{ and }  N_2 \leq \abs{k-k_1}, \abs{k-k_2} \leq 2N_2
\end{equation*} 
and we define $\psi(k,k_1,k_2):=\abs{k_1-k}^{2\alpha}+\abs{k_1}^{2\alpha}-\abs{k_2-k}^{2\alpha}-\abs{k_2}^{2\alpha}$. Let $\eta$ be a smooth cut-off function such that $\eta(t)=1$ for $\abs{t} \leq 1$ and $\eta(t)=0$ for $\abs{t}>2$ and define $\eta_T(t):=\eta(t/T)$. Integrating in time,
\begin{equation*}
\int_{[0,T] \times \T} \abs{S(t)\phi_1S(t)\phi_2}^2 \leq \sum_k \sum_{k_1,k_2}  \widehat{\phi}_1(k_1)\overline{\widehat{\phi}_1(k_2)}\widehat{\phi}_2(k-k_1)\overline{\widehat{\phi}_2(k-k_2)} \hat{\eta}_{T}(\psi(k,k_1,k_2)). 
\end{equation*}
The inner summation in $(k_1,k_2)$ is estimated using Schur's test, where for fixed $k \in \Z$, we view $\abs{\hat{\eta}_{T}(\psi(k,k_1,k_2))}$ as the kernel acting on the sequence $$k_1 \mapsto \abs{\widehat{\phi}_1(k_1)\widehat{\phi}_2(k-k_1)}$$ in $\ell^2$. Therefore, it suffices to bound
\begin{equation*}
    T\sup_{k, k_1} \sum_{k_2} \abs{\hat{\eta}\left(T\psi(k,k_1,k_2)\right)}.
\end{equation*}
which is equivalent to estimating
\begin{equation*}
    \sup_{\substack{J \subset \R \\ \abs{J} \lesssim 1}}\sup_{k, k_1 \in \Z}\# \left\{k_2 \in \Z: \abs{k_2} \sim N_1, \abs{k-k_2} \sim N_2, T\psi(k,k_1,k_2) \in J \right\}.
\end{equation*}
We can now apply Lemma \ref{Lemma: elementary counting lemma} with $f_k(x):=T\abs{x}^{2\alpha}+T\abs{x-k}^{2\alpha}$ for fixed $k\in \Z$. Note that without loss in generality we may assume that $\min\left(N_1, N_2\right) \gg 1$ (otherwise the counting problem is trivial).\\
For all $\abs{x} \sim N_1$ and $\abs{x-k} \sim N_2$, we have the estimate $ f''(x) \gtrsim TN_1^{2\alpha-2}$ and so,
\begin{equation*}
    \int_{[0,T] \times \T} \abs{S(t)\phi_1S(t)\phi_2}^2 \lesssim \left(T+T^\frac{1}{2}\min\left(N_1,N_2\right)^{1-\alpha}\right)\norm{\phi_1}_2^2 \norm{\phi_2}_2^2.
\end{equation*}
Now suppose that $4 \min\left(N_1,N_2\right)\leq \max\left(N_1,N_2\right)$. Then it is clear that for all $\abs{x} \sim N_1$ and $\abs{x-k} \sim N_2$, we have $\abs{f'(x)} \gtrsim T\max\left(N_1,N_2\right)^{2\alpha-1}$ which yields that,
\begin{equation*}
    \int_{[0,T] \times \T} \abs{S(t)\phi_1S(t)\phi_2}^2 \lesssim \left(T+\max\left(N_1,N_2\right)^{1-2\alpha}\right) \norm{\phi_1}_2^2 \norm{\phi_2}_2^2.
\end{equation*}
\end{proof}
\end{lemma}
\begin{remark}
Note that for $\phi \in L^2(\T)$ satisfying $\phi=P_N \phi$, Lemma \ref{lemma_Strichartz_bilinear}, together with Hölder's inequality, implies the bound
\begin{equation*}
    \norm{S(t)\phi}_{L^4_{t,x}\left([0,T] \times \T \right)} \lesssim \left(T+T^\frac{1}{2}N^{1-\alpha}\right)^\frac{1}{4} \norm{\phi}_2,
\end{equation*}
which represents an improvement by $N^\epsilon$ over the $L_{t,x}^4$ bound proved in \cite{demirbas_fractional}.
\end{remark} 
\begin{remark}
    We also stress that the constants in Lemma \ref{lemma_Strichartz_bilinear} are sharp for all $T>0$ when $N_1 \sim N_2$ and for large enough $T$ when $N_1 \gg N_2$. To this end, we only consider the case $\lambda=1$ as the sharp examples for $\lambda > 1$ follow from rescaling. If $\phi_j(x)=e^{iN_j x}$ for $j=1,2$ then clearly,
\begin{equation*}
    \int_{[0,T] \times \T} \abs{S(t)\phi_1(x) S(t) \phi_2(x)}^2 dx dt \sim T,
\end{equation*}
whereas $\norm{\phi_1}_2^2\norm{\phi_2}_2^2 \sim 1$ proving that the factor $T$ is sharp. Now consider the case $N_1 \gg N_2$ and $T\geq\left(N_1 N_2\right)^{1-2\alpha}$. For $j=1,2$, define $\phi_j(x)=\sum_{k=N_j}^{N_j+M_j}e^{ikx}$ where $M_1=\lfloor (N_1N_2)^\frac{2\alpha-1}{2}\rfloor$ and $M_2=\lfloor N_2^{2\alpha-1}\rfloor$. It follows that for $j=1,2$,
\begin{equation*}
    \abs{S(t)\phi_j(x)} = \abs{\sum_{k=0}^{M_j} e^{i\left((x+2\alpha N_j^{2\alpha-1}t)k+2\alpha(2\alpha-1)N_j^{2\alpha-2}k^2t+R_{\alpha}(k,N_j)t\right)}},
\end{equation*}
where the remainder term satisfies $\abs{R_\alpha(k,N_j)}\lesssim N_j^{2\alpha-3}k^3$. Therefore, by the uncertainty principle, $\abs{S(t)\phi_1(x)} \gtrsim M_j$ on $\abs{x+2\alpha N_j^{2\alpha-1}t}\lesssim M_j^{-1}$ and $t \lesssim M_j^{-2}$ and so
\begin{equation*}
    \int_{[0,T] \times \T} \abs{S(t)\phi_1(x) S(t) \phi_2(x)}^2 dx dt \gtrsim M_1^2 M_2^2 M_1^{-3} = N_1^\frac{1-2\alpha}{2} N_2^\frac{3(2\alpha-1)}{2},
\end{equation*}
whereas $\norm{\phi_1}_2^2\norm{\phi_2}_2^2 \sim N_1^\frac{2\alpha-1}{2} N_2^\frac{3(2\alpha-1)}{2} $ which shows that the term $N_1^{1-2\alpha}$ is sharp. When $T<\left(N_1 N_2\right)^{1-2\alpha}$, the preceding example yields a bound of the order of $TN_2^{2\alpha-1}$ which we conjecture to be sharp. However, we do not pursue this refinement in the present work. \\
When $N_1 \sim N_2 \sim N$ and $T<N^{2(1-\alpha)}$, we let $M_j=T^{-\frac{1}{2}}N^{1-\alpha}$ for $j=1,2$. An analogous argument to the one above proves that the constant $TN^{1-\alpha}$ is sharp.
\end{remark}
As we will be using the $I$-method to prove global well-posedness, we now recall the linear $L_{t,x}^6$ Strichartz estimate which will be used when bounding the growth of the modified energy (see Lemma \ref{lemma: growth E2}). The latter can be proved by a direct application of the discrete $L^2$-restriction Theorem \cite[Theorem 2.2]{Bourgain_Demeter_decoupling} - for example see \cite{Schippa_fractional}. 
\begin{lemma}\cite[Proposition 1.1]{Schippa_fractional}
    \label{lemma_Strichartz_Linear}
   Let $\phi \in L^2(\T_\lambda)$ be such that $\phi=P_N\phi$. Then, for every $\epsilon>0$, there exists a constant $C=C(\alpha,\epsilon)>0$ such that for any $0<T \leq \lambda^{2\alpha}$, we have the estimate
   \begin{equation*}
         \norm{S(t)\phi}_{L_{t,x}^6([0,T] \times \T_\lambda)} \leq C \lambda^{\epsilon} N^{\frac{1-\alpha}{3}+\epsilon} \norm{\phi}_2.
    \end{equation*} 
\end{lemma}
 We conclude this section by using Lemma \ref{lemma: transference multilinear X(s,b)} to transfer the linear and bilinear Strichartz estimates to $X^{s,b}_T$ spaces.
\begin{lemma}
\label{lemma:Linear X^0 Strichartz}
For every $\epsilon>0$ and $b>\frac{1}{2}$, there exists a constant $C=C(\alpha,b,\epsilon)>0$ such that for any $0<T\leq \lambda^{2\alpha}$, we have the following estimate
    \begin{equation*}
          \norm{u}_{L_{t,x}^6([0,T] \times \T_\lambda)}\leq C\lambda^{\epsilon}N^{\frac{1-\alpha}{3}+\epsilon} \norm{u}_{X^{0,b}_T}.
    \end{equation*} 
\end{lemma}    
\begin{lemma}
    \label{lemma: Bilinear X^0 Strichartz}
For every $b>\frac{1}{2}$ there exists a constant $C=C(\alpha,b)>0$ such that for any $T>0$ and $u_1, u_2 \in X^{0,b}_T$ satisfying $u_1 = P_{N_1} u_1$ and $u_2 = P_{N_2} u_2$, the following estimate holds
    \begin{equation*}
        \norm{u_1 u_2}_{L_{t,x}^2([0,T] \times \T_\lambda)} \leq C  \left(\frac{T}{\lambda}+T^\frac{1}{2}\min\left(N_1,N_2\right)^{1-\alpha}\right)^{\frac{1}{2}}\norm{u_1}_{X^{0,b}_T} \norm{u_2}_{X^{0.b}_T}.      
    \end{equation*}
Moreover, if $\max\left(N_1,N_2\right) \geq 4\min\left(N_1,N_2\right) $, the following refined estimate holds
    \begin{equation*}
        \norm{u_1 u_2}_{L_{t,x}^2([0,T] \times \T_\lambda)} \leq C  \left(\frac{T}{\lambda}+\max\left(N_1,N_2\right)^{1-2\alpha}\right)^{\frac{1}{2}}\norm{u_1}_{X^{0,b}_T} \norm{u_2}_{X^{0.b}_T}.      
    \end{equation*}
\end{lemma}
When $s > \tfrac{1-\alpha}{2}$, the bilinear estimates in Lemma~\ref{lemma: Bilinear X^0 Strichartz} suffice to obtain local well-posedness in $X^{s,b}_T$ by an argument analogous to \cite[Theorem~3]{Burq_Gerard_Tzvetkov}. However, for the endpoint case $s = \tfrac{1-\alpha}{2}$, these estimates are not strong enough in $b$, since the argument in \cite[Theorem~3]{Burq_Gerard_Tzvetkov} relies on interpolation to enlarge the admissible range of $b$ at the expense of spatial regularity, $s$. In our bilinear estimates, the condition $b > \tfrac{1}{2}$ arises from Minkowski's inequality (see the proof of Lemma~\ref{lemma: transference multilinear X(s,b)}). This condition can be improved by establishing an inhomogeneous bilinear Strichartz estimate via a counting argument (see the proof of Lemma~\ref{lemma: Improved_in_b_Bilinear X^0 Strichartz}). Since local well-posedness estimates will ultimately be proved for times $T \leq \lambda^{2}$ (see Lemma~\ref{lemma: lwp_ILNS}), we restrict attention to this regime in the next lemma.
\begin{lemma}
    \label{lemma: Improved_in_b_Bilinear X^0 Strichartz}
For every $\epsilon>0$ there exists a constant $C=C(\alpha,\epsilon)>0$ such that for any $0<T\leq \lambda^2$, $u_1 \in X^{0,\frac{1}{2}+\epsilon}_T$ and $u_2 \in X^{0,\frac{1}{4}}_T$ satisfying $u_1 = P_{N_1} u_1$ and $u_2 = P_{N_2} u_2$, the following estimate holds
    \begin{equation}
        \norm{u_1 u_2}_{L_{t,x}^2([0,T] \times \T_\lambda)} \leq C T^\frac{1}{4}\min\left(N_1,N_2\right)^{\frac{1-\alpha}{2}}\norm{u_1}_{X^{0,\frac{1}{2}+\epsilon}_T} \norm{u_2}_{X^{0,\frac{1}{4}}_T}.   
        \label{eq: bilinear_strichartz_impoved in b}
    \end{equation}
Moreover, if $\max\left(N_1,N_2\right) \geq 4\min\left(N_1,N_2\right)$, we have the following estimate
    \begin{equation}
        \norm{u_1 u_2}_{L_{t,x}^2([0,T] \times \T_\lambda)} \leq C T^\frac{1}{4}\min\left(N_1,N_2\right)^{O(\epsilon)}\norm{u_1}_{X^{0,\frac{1}{2}+\epsilon}_T} \norm{u_2}_{X^{0,\frac{1}{2}-\epsilon}_T}.   
        \label{eq: high_low bilinear_strichartz_impoved_b_improvement}
    \end{equation}
\begin{proof}
We begin by proving the following estimate
\begin{equation}
    \norm{S(t) \phi_1 S(t) f_2}_{L_{t,x}^2([0,T] \times \T_\lambda)} \lesssim T^\frac{1}{4}\min\left(N_1,N_2\right)^{\frac{1-\alpha}{2}}\norm{u_1}_{2} \norm{f_2}_{H^{\frac{1}{4}}_T},
    \label{eq: inhomogenous_strichartz_bilinear}
\end{equation}
where $0<T \leq \lambda^2$, $\phi_1 \in L^2(\T_\lambda)$ satisfies $\phi_1 = P_{N_1} \phi_1$ and $f_2 \in H_T^{\frac{1}{4}}\left(\R \rightarrow L^2(\T_\lambda) \right)$ satisfies $f_2=P_{N_2} f_2$. By Littlewood-Paley theory and by rescaling it suffices to prove
\begin{equation*}
    \norm{S(t) \phi_1 S(t) f_2}_{L_{t,x}^2([0,T] \times \T)}^2 \lesssim T^{-\frac{1}{2}}L^{\frac{1}{2}}\min\left(N_1,N_2\right)^{1-\alpha}\norm{\phi_1}_{2}^2 \norm{f_2}_{2}^2,
\end{equation*}
where $0<T\leq 1$, $\phi_1 \in L^2(\T)$ satisfies $\phi_1 = P_{N_1} \phi_1$ and $f_2 \in L^2(\R \times \T)$ satisfies $$\supp{ \tilde{f}_2} \subset \left\{(\omega,k) \in \R \times \Z : \abs{\omega} \lesssim LT^{-1}, \abs{k} \sim N_2  \right\}.$$
The proof \eqref{eq: inhomogenous_strichartz_bilinear} is then analogous to that of Lemma \ref{lemma_Strichartz_bilinear} but now, for all $0<T\leq 1 $, we need to bound
\begin{equation}
     \sup_{k, k_1} \int_{\abs{\omega} \lesssim L T^{-1} } \sum_{\substack{\abs{k_2} \sim N_1 \\ \abs{k-k_2} \sim N_2}} \abs{\hat{\eta}_T(\psi(k,k_1,k_2)-\omega)} d\omega,
     \label{eq: desired_bound_for_bilinear}
\end{equation}
where $\psi(k,k_1,k_2):=\abs{k_1-k}^{2\alpha}+\abs{k_1}^{2\alpha}-\abs{k_2-k}^{2\alpha}-\abs{k_2}^{2\alpha} $ and $\eta_T$ is a smooth cut-off function adapted to the interval $[-T,T]$. This is done via a counting argument
\begin{align*}
    \eqref{eq: desired_bound_for_bilinear} &\lesssim T \sup_{\substack{J \subset \R \\ \abs{J} \lesssim 1}}\sup_{k, k_1 \in \Z} \sum_{\substack{\abs{k_2} \sim N_1 \\ \abs{k-k_2} \sim N_2}} \int_{\substack{\abs{\omega} \lesssim L T^{-1} \\ T\left(\psi(k,k_1,k_2)-\omega\right)  \in J }  } \, \, d\omega, \\
    &\lesssim \sup_{\substack{J \subset \R \\ \abs{J} \lesssim 1}}\sup_{k, k_1 \in \Z} \#\left\{ k_2 \in  \Z: \abs{k_2} \sim N_1,  \abs{k-k_2} \sim N_2, L^{-1} T \psi(k,k_1,k_2) \in J \right\}, \\
    &\lesssim T^{-\frac{1}{2}}L^{\frac{1}{2}} \min\left(N_1,N_2\right)^{1-\alpha},
\end{align*}
where in the last line we used Lemma \ref{Lemma: elementary counting lemma} with $f_k(x):=L^{-1} T \abs{x}^{2\alpha}+L^{-1} T\abs{k-x}^{2\alpha}$ for fixed $k \in \Z$. Next, we show how \eqref{eq: bilinear_strichartz_impoved in b} is obtained from \eqref{eq: inhomogenous_strichartz_bilinear}. Suppose that $u_1$ and $u_2$ satisfy the assumptions of the lemma. Define $f_2:=S(-t)u_2$ and write (see Lemma \ref{lemma: transference multilinear X(s,b)}),
\begin{align*}
    \phi_{\omega}(x) &:= \int_{\Z_\lambda} \widetilde{u}_1(\omega+\abs{k}^{2\alpha},k) e^{ikx}\left(dk_\lambda\right), \\
    u_1(t)&\phantom{:}=\frac{1}{2\pi} \int_\R e^{i\omega t}S(t) \phi_{\omega} d\omega.
\end{align*}
Then, by Minkowski's inequality,
\begin{equation*}
    \norm{u_1 u_2 }_{L_{t,x}^2([0,T] \times \T_\lambda)} \leq \int_{-\infty} ^\infty \norm{S(t)\phi_{\omega}S(t)f_2}_{L_{t,x}^2([0,T] \times \T_\lambda)}\, \, d\omega.
\end{equation*}
Applying \eqref{eq: inhomogenous_strichartz_bilinear} and recalling that $u_1 \in X_T^{0,\frac{1}{2}+}$, we obtain \eqref{eq: bilinear_strichartz_impoved in b}. \\
We conclude with the proof of \eqref{eq: high_low bilinear_strichartz_impoved_b_improvement}. Note that if $\max\left(N_1, N_2 \right) \geq 4\min\left(N_1,N_2\right)$, then Lemma \ref{lemma: Bilinear X^0 Strichartz} implies that for $T \leq \lambda^2$ 
\begin{equation*}
        \norm{u_1 u_2}_{L_{t,x}^2([0,T] \times \T_\lambda)} \lesssim T^\frac{1}{4}\norm{u_1}_{X^{0,\frac{1}{2}+\epsilon}_T} \norm{u_2}_{X^{0,\frac{1}{2}+\epsilon}_T}.   
        \label{eq: high_low bilinear_strichartz_impoved in b}
\end{equation*}
Interpolating this estimate with \eqref{eq: bilinear_strichartz_impoved in b}, yields \eqref{eq: high_low bilinear_strichartz_impoved_b_improvement}.
\end{proof}
\end{lemma}
\section{Proof of Theorem \ref{theorem:main_gwp_theorem_n=1}}\label{sec:4}
 The aim of this section is to prove Theorem \ref{theorem:main_gwp_theorem_n=1} using the $I$-method. Henceforth, let $\frac{1-\alpha}{2}\leq s<\alpha$ and unless otherwise stated all spatial norms are defined on $\T_\lambda$, i.e. $H^s=H^s(\T_\lambda)$, etc. For some $N \gg 1$ to be established later, we recall that $m:=m_N^{\alpha-s}$ and $I:=I_N^{\alpha-s}$, where $m_N$ and $I_N$ are defined in \cref{sec: I-operator}. Next, we define the first modified energy
\begin{equation*}
    E^1\left(u\right):= E\left(Iu\right)= \frac{1}{2}\int_{\T_\lambda} \abs{\left(-\Delta\right)^\frac{\alpha}{2}Iu(x)}^2 \,\, dx + \frac{1}{4} \int_{\T_\lambda} \abs{Iu(x)}^{4}\,\, dx.
\end{equation*}
As observed in \cite{KdV}, a hierarchy of modified energies can be formally defined for dispersive nonlinear equations. We will use the second modified energy and quantify its growth as a function of $N$.
To do this, it is convenient to introduce some notation. For $n \geq 2$, even, define a spatial multiplier of order $n$ to be a function $M_n(k_1,\dots, k_n)$ on $\Gamma_n=\left\{ (k_1, \dots, k_n) \in \Z_\lambda^n : k_1+\cdots+k_n=0 \right\}$ which is endowed with measure $\prod_{j=1}^{n-1} (dk_j)_\lambda$. Additionally for functions $f_1, \dots,f_n$ on $\T_\lambda$, define
\begin{equation*}
    \Lambda_n(M_n; f_1, \dots,f_n) := \int_{\Gamma_n} M_n(k_1,\dots, k_n) \prod_{j=1}^n \hat{f}_j(k_j).
\end{equation*}
We will often write $\Lambda_n(M_n;f)$ instead of $\Lambda_n(M_n; f, \overline{f}, \dots, f, \overline{f})$. Using this notation, the first modified energy becomes
\begin{equation*}
    E^1(u)=\frac{1}{2} \Lambda_2(m(k_1)\abs{k_1}^\alpha m(k_2) \abs{k_2}^\alpha;u)+ \frac{1}{4}\Lambda_4(m(k_1)m(k_2)m(k_3)m(k_4);u).
\end{equation*}
 We can now define the second modified energy
\begin{equation*}
    E^2(u):= \frac{1}{2} \Lambda_2(m(k_1)\abs{k_1}^\alpha m(k_2) \abs{k_2}^\alpha;u)+ \frac{1}{4}\Lambda_4(M_4(k_1,k_2,k_3,k_4);u),
\end{equation*}
where
\begin{equation}
\label{def: M4}
    M_4(k_1,k_2,k_3,k_4) := 
\begin{cases}
  \dfrac{\sum_{j=1}^4 (-1)^{j} \abs{k_j}^{2\alpha} m(k_j)^2}
        {\abs{k_1}^{2\alpha}-\abs{k_2}^{2\alpha}+\abs{k_3}^{2\alpha}-\abs{k_4}^{2\alpha}},
    & (k_1,k_2,k_3,k_4) \notin \Gamma_{\mathrm{res}}, \\[1ex]
  m(k_1)m(k_2)m(k_3)m(k_4),
    & (k_1,k_2,k_3,k_4) \in \Gamma_{\mathrm{res}},
\end{cases}
\end{equation}
where $\Gamma_{\text{res}}$ denotes the set of resonant frequencies,
\begin{align*}
    \Gamma_{\text{res}}&:=\{(k_1,k_2,k_3,k_4) \in \Gamma_4: \abs{k_1}^{2\alpha}-\abs{k_2}^{2\alpha}+\abs{k_3}^{2\alpha}-\abs{k_4}^{2\alpha}=0 \}, \\
    &\phantom{:}=\{(k_1,k_2,k_3,k_4) \in \Gamma_4: (k_1+k_2)(k_2+k_3)=0 \}.
\end{align*}
If $u: \mathbb{T}_\lambda \times \R \rightarrow \mathbb{C}$ is a smooth solution to (\ref{eq:fractional_main_lwp}), then
\begin{equation*}
    \partial_t E^2(u(t))= i\Lambda_6(M_6;u(t)),
\end{equation*}
where,
\begin{align}
M_6(k_1,\dots,k_6):=M_4(k_1+k_2+k_3,k_4,k_5,k_6)-M_4(k_1,k_2+k_3+k_4,k_5,k_6) \nonumber \\
    +M_4(k_1,k_2,k_3+k_4+k_5,k_6)-M_4(k_1,k_2,k_3,k_4+k_5+k_6) .
    \label{eq: M6}
\end{align}
We note that, by construction, bounding the growth of $E^2$ is expected to be more straightforward than that of $E^1$ since $M_4$ is defined in such a way that it cancels the effect of $\Lambda_2(m(k_1)\abs{k_1}^\alpha m(k_2) \abs{k_2}^\alpha;u(t))$ when taking the time derivative. This is only possible because the resonant set is trivial and coincides with the resonant set of the NLS ($\alpha=1$) equation which is integrable. Indeed, when the denominator of $M_4$ vanishes, the numerator also vanishes, which allows us to establish a bound for $M_4$ (see Lemma \ref{lemma: M_4_bounded}). To this end, we first recall a lemma from \cite{demirbas_fractional} which characterizes the resonant set.
\begin{lemma}\cite[Lemma 2]{demirbas_fractional}
\label{lemma: convexity_demirbas}
    There exists a constant $C=C(\alpha)>0$ such that
    \begin{equation*}
    \left\vert\abs{k_1}^{2\alpha}-\abs{k_2}^{2\alpha}+\abs{k_3}^{2\alpha}-\abs{k_4}^{2\alpha}\right\rvert \geq C\frac{\abs{k_1+k_2}\abs{k_2+k_3}}{\left(\abs{k_1}+\abs{k_2}+\abs{k_3}+\abs{k_4}\right)^{2-2\alpha}},
\end{equation*}
for all $(k_1,k_2,k_3,k_4) \in \Gamma_4$.
\end{lemma}
As a consequence, we obtain the following bound for $M_4$.
\begin{lemma}
\label{lemma: M_4_bounded}
There exists a constant $C=C(\alpha,s)>0$ such that
\begin{equation*}
    \abs{M_4(k_1,k_2,k_3,k_4)} \leq Cm(k_3^\star)^2
\end{equation*}
for all $(k_1,k_2,k_3,k_4) \in \Gamma_4$, where \( (k_1^\star, k_2^\star, k_3^\star, k_4^\star) \) denotes the ordering of the tuple \( (k_1, k_2, k_3, k_4) \) such that \( \abs{k_1^\star} \geq \abs{k_2^\star} \geq \abs{k_3^\star} \geq \abs{k_4^\star} \).
\begin{proof}
    For convenience, define $f(t):= \abs{t}^{2\alpha}m(t)^2$ and $g(t):=\abs{t}^{2\alpha}$. By homogeneity, it suffices to consider $N=1$. If $\abs{k_1^\star}\leq 1$, then the estimate follows immediately, so we restrict our attention to $\abs{k_1^\star}>1$.  By symmetry, we may assume that  $k_1\geq k_3\geq 0\geq k_4\geq k_2$ and without loss in generality, we set $k_1=k_1^\star$.\\
    First, consider the case $k_3 \geq k_3^\star$. It follows immediately that $(k_1,k_2,k_3,k_4) \in \Gamma_{\text{res}}$ and so, by construction, $M_4$ satisfies the desired bound. Therefore, we only need to consider the case $k_3=k_4^\star$ or equivalently $k_1 \geq -k_2 \geq -k_4 \geq k_3 \geq 0$. Since $(k_1,k_2,k_3,k_4) \in \Gamma_4$, we have $k_1 \sim \abs{k_2}$.  We now consider three subcases:
    \begin{enumerate}
        \item[(i)] $\abs{k_2+k_3} \gtrsim k_1$ and $k_3 \sim \abs{k_4}$, 
        \item[(ii)] $\abs{k_2+k_3} \gtrsim k_1$ and $k_3 \ll \abs{k_4}$,
        \item[(iii)] $\abs{k_2+k_3} \ll k_1$.
    \end{enumerate}
For (i) and (ii), the denominator of $M_4$ is lower bounded using Lemma \ref{lemma: convexity_demirbas},
\begin{equation*}
    \abs{g(k_1)-g(k_2)+g(k_3)-g(k_4)} \gtrsim (k_1+k_2)k_1^{2\alpha-1}.
\end{equation*}
For (i), we use the triangle inequality followed by the mean value theorem to bound the numerator of $M_4$,
\begin{align*}
            \abs{f(k_1)-f(k_2)+f(k_3)-f(k_4)} &\leq \abs{f(k_1)-f(k_2)}+ \abs{f(k_3)-f(k_4)}, \\
            &\lesssim (k_1+k_2) m(k_4)^2 k_1^{2\alpha-1}.
\end{align*}
For (ii), we have that $k_3 \ll \abs{k_4}$ and so $k_1+k_2 = \abs{k_3+k_4} \gtrsim \abs{k_4}$. Therefore,
\begin{align*}
            \abs{f(k_1)-f(k_2)+f(k_3)-f(k_4)} &\leq \abs{f(k_1)-f(k_2)} +2f(k_4), \\
            &\lesssim (k_1+k_2) m(k_4)^2 k_1^{2\alpha-1}
\end{align*}
For (iii), note that $k_1 \sim k_3 \sim \abs{k_2} \sim \abs{k_4}$. 
The denominator of $M_4$ is bounded using  Lemma \ref{lemma: convexity_demirbas},
\begin{equation*}
    \abs{g(k_1)-g(k_2)-g(k_3)-g(k_4)} \gtrsim (k_1+k_2)\abs{k_2+k_3}k_1^{2\alpha-2},
\end{equation*}
whereas the numerator of $M_4$ can be bounded by the double-mean value theorem,
\begin{align*}
            \abs{f(k_1)-f(k_2)+f(k_3)-f(k_4)} &\leq \abs{\int_{-k_2}^{k_1} \int_{k_3+k_2}^0 f''(\eta+\xi) d\eta d\xi },\\
            &\lesssim (k_1+k_2) \abs{k_2+k_3} m(k_3)^2 k_1^{2\alpha-2}. \qedhere
        \end{align*}
\end{proof}
\end{lemma}
 We now turn to the proof of Theorem  \ref{theorem:main_gwp_theorem_n=1}. Rather than working directly with the original \eqref{eq:fractional_main_lwp} on $\T$, we prefer to work with solutions $Iu $ to \eqref{eq:fractional_I_system} (see \cref{sec:4.1}) on $\T_\lambda$ in order to exploit the scaling symmetry of the equation. Specifically, we continue solutions $Iu$ of the \eqref{eq:fractional_I_system} which yields a growth bound for solutions of the \eqref{eq:fractional_main_lwp} on the rescaled torus $\T_\lambda$. By inverting the rescaling, we obtaining polynomial in time bounds for the $H^s(\T)$ norm of solutions to the \eqref{eq:fractional_main_lwp} thereby establishing global well-posedness. Following \cite{mass_critical_gwp_staffilani_tzirakis}, we split the proof of Theorem \ref{theorem:main_gwp_theorem_n=1} in three steps: (i) establishing local in time bounds for the $H^\alpha$ norm of solutions to (\ref{eq:fractional_I_system}), (ii) demonstrating that $E^2$ controls the $H^\alpha$ norm of $Iu$ by showing it is a good approximation of $E^1$ and (iii) deriving a growth bound for $E^2$.
\subsection{Local (in time) bounds for solutions to (I-fNLS)}\label{sec:4.1}
The first step in the proof of Theorem \ref{theorem:main_gwp_theorem_n=1} is establishing a local in time bound for $\norm{Iu}_{X^{\alpha,b}_T}$ where $u$ is a solution to (fNLS). To do this we apply $I$-operator to (fNLS) and use a fixed point argument on the following initial value problem, 
\begin{equation}
    \label{eq:fractional_I_system}
    \begin{cases}
       i \partial_t Iu+\left(-\Delta\right)^{\alpha}Iu=-I\left(\left\lvert u \right \rvert^2 u\right), \\
       Iu(x,0)=Iu_0(x) \in H^{\alpha}.
    \end{cases}
    \tag{I-fNLS}
\end{equation}
In particular, we prove the following result.
\begin{lemma}
\label{lemma: lwp_ILNS}
For any $s\geq \frac{1-\alpha}{2}$, $r>0$ and $0<\epsilon \ll 1$, there exist constants $C=C(s,\alpha,r,\epsilon)>0$  and $\lambda_0=\lambda_0(s,\alpha,r,\epsilon)>0$ such that for any $\lambda\geq \lambda_0$ and any initial data $u_0$ satisfying $\norm{Iu_0}_{H^\alpha} \leq r \lambda^{\frac{1-2\alpha}{2}}$, we have the following estimate for solutions to \eqref{eq:fractional_main_lwp}
\begin{equation*}
   \norm{Iu}_{X^{\alpha,\frac{1}{2}+\epsilon}_T\left([0,T]\right)} \leq C\lambda^{\frac{1-2\alpha}{2}},
\end{equation*}
where $T=\lambda^{2(2\alpha-1)-\epsilon}$.
\end{lemma}
 We remark that the condition $\norm{Iu_0}_{H^\alpha} \leq r \lambda^{\frac{1-2\alpha}{2}}$ stems from the mass subcritical nature of the equation (see the proof of Theorem \ref{theorem:main_gwp_theorem_n=1}). By standard iterative methods and the properties of $X^{\alpha,b}_T$ spaces, Lemma \ref{lemma: lwp_ILNS} follows immediately from the following trilinear estimate.
\begin{lemma}
\label{lemma:trilinear_estimates_lwp}
For any $s \geq \frac{1-\alpha}{2}$ and $0<\epsilon \ll 1$ there exist a constant $C=C(s,\alpha,\epsilon)>0$ such that for any $0<T \leq \lambda^2$, we have the estimate
    \begin{equation*}
    \norm{\int_0^t S(t-t')I(u_1\overline{u}_2u_3) dt'}_{X^{\alpha,\frac{1}{2}+\epsilon}_T\left([0,T]\right)} \leq C T^{\frac{1}{2}} \prod_{j=1}^3 \norm{Iu_j}_{X^{\alpha,\frac{1}{2}+\epsilon}_T\left([0,T]\right)}.
\end{equation*}
\begin{proof}
First note that by Lemma \ref{lemma: I-method_invariant_lemma}, it suffices to show that
    \begin{equation*}
\norm{\int_0^t S(t-t')(u_1\overline{u}_2u_3) dt'}_{X^{s,\frac{1}{2}+}_T\left([0,T]\right)} \lesssim T^{\frac{1}{2}} \prod_{j=1}^3 \norm{u_j}_{X^{s,\frac{1}{2}+}_T\left([0,T]\right)}.
\end{equation*}
By Lemma \ref{lemma: X(s,b) properties}(b), we have that
        \begin{equation*}
            \norm{\int_0^t S(t-t') u_1\overline{u}_2u_3  \,\,dt'}_{X^{s,\frac{1}{2}+}_T\left([0,T]\right)} \lesssim T \norm{u_1u_2u_3}_{X^{s,-\frac{1}{2}+}_T\left([0,T]\right)},
        \end{equation*}
so, by duality, it suffices to show
\begin{equation*}
    \abs{\int_0^T \int_{\T_\lambda} u_1 \overline{u}_2 u_3 \overline{v} dx dt} \lesssim T^{\frac{1}{2}} \norm{v}_{X^{-s,\frac{1}{2}-}_T} \prod_{j=1}^3 \norm{u_j}_{X^{s,\frac{1}{2}+}_T}.
\end{equation*}
We dyadically decompose $v=\sum_{M \in 2^{\N_0}} P_{M}v$ and for $j=1,2,3$, $u_j=\sum_{N_j \in 2^{\N_0}} P_{N_j}u_j$. Note that by orthogonality, we have that $M \lesssim N_1+N_2+N_3$. Therefore, we need to bound
\begin{equation*}
    \sum_{M,N} I(M,N) := \sum_{M,N} \int_0^T\int_{\T_\lambda} \abs{P_{N_1}u_1P_{N_2}u_2P_{N_3}u_3P_{M}v},
\end{equation*}
where $N:=(N_1,N_2,N_3)$ and $M \lesssim \max\left(N_1,N_2,N_3\right)$. Without loss in generality we can assume that $N_1 \geq N_2 \geq N_3$ and so and we can split the summation into two regions: (i) $N_1 \geq N_2 \geq N_3$ with $N_1 \sim M$ and (ii) $N_1 \sim N_2 \geq N_3$ with $N_1 \gg M$. \\
We begin by recalling that if we combine estimates \eqref{eq: bilinear_strichartz_impoved in b} and \eqref{eq: high_low bilinear_strichartz_impoved_b_improvement} from Lemma \ref{lemma: Improved_in_b_Bilinear X^0 Strichartz}, we have
\begin{equation*}
    \norm{P_{L_1}f_1 P_{L_2}f_2}_{L^2_{t,x}\left([0,T]\times \T_\lambda\right)} \lesssim T^{\frac{1}{4}} C_0(L_1,L_2) \norm{P_{L_1}f_1}_{X^{0,\frac{1}{2}+}_T} \norm{P_{L_2}f_2}_{X^{0,\frac{1}{2}-}_T},
\end{equation*}
where $L_1, L_2 \in 2^{\N_0}$ and
\begin{equation*}
    C_0(L_1,L_2):= \begin{cases}
       L_1^{\frac{1-\alpha}{2}} &\text{ if } L_1 \sim L_2 \\
        \min \left(L_1,L_2\right)^{0+} &\text{ otherwise}
    \end{cases}
\end{equation*}
Therefore, for case (i), we have the following estimate
\begin{equation*}
    \sum_{\substack{N\\ N_1 \sim M }} I(M,N) \lesssim T^\frac{1}{2} \sum_{\substack{N \\ N_1 \sim M }}\frac{C_0(N_1,N_2) C_0(M,N_3)}{\left(N_1 N_2 N_3 \right)^sM^{-s}}   \norm{P_M v}_{X^{-s,\frac{1}{2}-}_T} \prod_{j=1}^3 \norm{P_{N_j}u_j}_{X^{s,\frac{1}{2}+}_T}.
\end{equation*}
We further break down the summation into three regions: (a) $N_1 \gg N_2$, (b) $N_1 \sim N_2 \gg N_3$, (c) $N_1 \sim N_2 \sim N_3$. For case (a), since $s\geq \frac{1-\alpha}{2}>0$, there exists a $\gamma>0$ such that 
\begin{equation}
\label{eq: lwp_multiplier}
   \frac{C_0(N_1,N_2) C_0(M,N_3)}{\left(N_1 N_2 N_3 \right)^sM^{-s}}  \lesssim \frac{1}{\left(N_2N_3\right)^\gamma}.
\end{equation}
Thus, by Cauchy–Schwarz
\begin{align*}
    \sum_{\substack{N\\ N_1 \sim M }} I(M,N)   &\lesssim T^{\frac{1}{2}} \prod_{j=2}^3\norm{u_j}_{X^{s,\frac{1}{2}+}_T} \sum_{N_1 \sim M}  \norm{P_{N_1}u_1}_{X^{s,\frac{1}{2}+}_T} \norm{P_Mv}_{X^{-s,\frac{1}{2}-}_T}  , \\
    &\lesssim  T^{\frac{1}{2}} \norm{v}_{X^{-s,\frac{1}{2}-}_T }\prod_{j=1}^3 \norm{u_j}_{X^{s,\frac{1}{2}+}_T}.
\end{align*}
The summation in regions (b) and (c) can be handled in a similar fashion. In particular, the multiplier on the left hand side of \eqref{eq: lwp_multiplier} is bounded $N_3^{-\gamma}$ for some $\gamma>0$ in case (b) and by an absolute constant in case (c). \\
Now for case (ii): $N_1 \sim N_2 \gg M$. Here we further split the summation into two regions (a) $N_1 \sim N_2 \gg M$ with $N_2 \gg N_3$ and (b) $N_1 \sim N_2 \sim N_3 \gg M$. We begin with case (a),
\begin{align*}
    \sum_{\substack{N \\ N_1 \sim N_2 \\ N_1\gg M }} I(M,N) &\lesssim T^{\frac{1}{2}} \sum_{\substack{N\\ N_1 \sim N_2 \\ N_1 \gg M }} \frac{C_0(N_1,M)C_0(N_2,N_3)}{\left(N_1 N_2 N_3 \right)^sM^{-s}}  \norm{P_M v}_{X^{-s,\frac{1}{2}-}_T} \prod_{j=1}^3 \norm{P_{N_j}u_j}_{X^{s,\frac{1}{2}+}_T}, \\
    &\lesssim T^{\frac{1}{2}} \sum_{\substack{N\\ N_1 \sim N_2 \\ N_1 \gg M }} \frac{1}{N_3^\gamma} \frac{M^s}{N_1^s} \norm{P_M v}_{X^{-s,\frac{1}{2}-}_T} \prod_{j=1}^3 \norm{P_{N_j}u_j}_{X^{s,\frac{1}{2}+}_T} , \\
     &\lesssim T^{\frac{1}{2}} \norm{u_3}_{X^{s,\frac{1}{2}+}_T} \sum_{\substack{N_1 \sim N_2 \\ N_1 \gg M }} \frac{M^s}{N_1^s}  \norm{P_M v}_{X^{-s,\frac{1}{2}-}_T} \prod_{j=1}^2 \norm{P_{N_j}u_j}_{{X^{s,\frac{1}{2}+}_T}}, \\
    &\lesssim T^{\frac{1}{2}} \prod_{j=2}^3 \norm{u_j}_{{X^{s,\frac{1}{2}+}_T}} \sum_{\substack{ N_1 \gg M }} \frac{M^s}{N_1^s} \norm{P_{N_1}u_1}_{{X^{s,\frac{1}{2}+}_T}} \norm{P_M v}_{X^{-s,\frac{1}{2}-}_T}, \\
    &\lesssim T^{\frac{1}{2}} \norm{v}_{X^{-s,\frac{1}{2}-}_T} \prod_{j=1}^3 \norm{u_j}_{X^{s,\frac{1}{2}+}_T},
\end{align*}
where in the second estimate we used the fact that $s\geq \frac{1-\alpha}{2}>0$. We conclude by noting that case (b) can be handled in a analogous manner.
\end{proof}
\end{lemma}
\subsection{$E^2$ is a small perturbation of $E^1$}
It is clear how $E^1$ can bound the $H^\alpha$ norm of $Iu$, however since we will be controlling the growth of $E^2$ (which is not sign definite), we show that $E^2$ is well approximated by $E^1$. 
\begin{lemma}
    \label{lemma: E_2_vs_E_1}
There exists a constant $C=C(\alpha,s)>0$ such that the following estimate holds for any $u \in H^s(\T_\lambda)$,
\begin{equation*}
    \abs{E^2(u)-E^1(u)} \leq C N^{-\alpha} \norm{Iu}_{H^\alpha}^4.
\end{equation*}
\begin{proof}
We dyadically decompose $u = \sum_{j\geq 1} u_j$ where $u_j:=P_{N_j} u$ for dyadic $N_j \in 2^{\N_0}$ and without loss in generality assume that $\hat{u}$ is real and non-negative. By the definition of $E^1$ and $E^2$, it suffices to prove that
\begin{equation*}
   \sum_{N_1,N_2 N_3 , N_4} \abs{\int_{\Gamma_4} \left(M_4- \prod_{j=1}^4 m(k_j)\right)\prod_{j=1}^4\hat{u}_j(k_j)} \lesssim N^{-\alpha} \norm{Iu}_{H^\alpha}^4.
\end{equation*}
We define $N_1^\star \geq N_2^\star \geq N_3^\star \geq N_4^\star$ to be the decreasing ordering of $(N_1, N_2,N_3, N_4)$ and set $u_j^\star:=P_{N_j^\star} u$. Observe that the multiplier $M_4-\prod_{j=1}^4 m(k_j)$ vanishes unless $N_1^\star \gtrsim N$. On $\Gamma_4$, this implies that the two largest frequencies satisfy $N_1^\star \sim N_2^\star  \gtrsim N$. Therefore, using Lemma \ref{lemma: M_4_bounded} to bound $M_4-\prod_{j=1}^4 m(k_j)$, it suffices to show
\begin{equation*}
    m(N_3^\star)^2\int_{\Gamma_4} \prod_{j=1}^4\hat{u}_j(k_j) \lesssim N^{-\alpha} \left(N_1^\star \right)^{0-}\prod_{j=1}^4 \norm{Iu_j}_{H^{\alpha}}.
\end{equation*}
This follows immediately from Sobolev embedding 
\begin{align*}
     m(N_3^\star)^2\int_{\Gamma_4} \prod_{j=1}^4\hat{u}_j(k_j) &\lesssim \frac{1}{m(N_1^\star)m(N_2^\star)} \int_{\Gamma_4} \prod_{j=1}^4 m(N_j^\star)\hat{u}_j(k_j), \\
     &\lesssim  \frac{1}{(N_1^\star)^\alpha m(N_1^\star)(N_2^\star)^\alpha m(N_2^\star)} \prod_{j=1}^2 \norm{Iu_j^\star}_{H^\alpha}\prod_{j=3}^4 \norm{Iu_j^\star}_{L^\infty},  \\
    &\lesssim \left(N_1^\star\right)^{0-}N^{-2\alpha+}\prod_{j=1}^2 \norm{Iu_j^\star}_{H^\alpha}\prod_{j=3}^4 \norm{Iu_j^\star}_{L^\infty}, \\
    &\lesssim  \left(N_1^\star\right)^{0-}N^{-2\alpha+} \left(N_3^\star\right)^{\frac{1}{2}-\alpha}\left(N_4^\star\right)^{\frac{1}{2}-\alpha} \prod_{j=1}^4 \norm{Iu_j}_{H^\alpha}. \qedhere
\end{align*}
\end{proof}
\end{lemma}
\subsection{Growth of $E^2$}
We estimate the growth of $E^2$ by using the long-time linear and bilinear Strichartz estimates from Lemma \ref{lemma:Linear X^0 Strichartz} and Lemma \ref{lemma: Bilinear X^0 Strichartz} respectively. We emphasize that the bilinear estimates are only required in the regime of small $\alpha$.
\begin{lemma}
    \label{lemma: growth E2}
Suppose that $s>\frac{1-\alpha}{3}$ and $u:\T_\lambda\times \R \rightarrow \mathbb{C}$ is a smooth solution to \eqref{eq:fractional_main_lwp}. Then, for any $t_0 \in \R $ and $0<\epsilon \ll 1$, there exists a constant $C=C(\alpha,s,\epsilon)>0$ such that
\begin{equation*}
    E^2(u(t_0+T)) -E^2(u(t_0)) \leq C \lambda^\epsilon N^{-\gamma(4\alpha-1)+\epsilon}\norm{Iu}_{X^{\alpha,\frac{1}{2}+\epsilon}_T\left([t_0,t_0+T]\right)}^6,
\end{equation*}
where $T=\lambda^{2(2\alpha-1)-\epsilon}$ and 
\begin{equation*}
    \gamma:=\begin{cases}
        \frac{2}{3} &\text{ if } \lambda^{4\alpha-3} \geq N^{-(2\alpha-1)} , \\
       1 &\text{ if } \lambda^{4\alpha-3} < N^{-(2\alpha-1)}.
    \end{cases}
\end{equation*}
\begin{proof}
It suffices to consider $t_0=0$. By the fundamental theorem of calculus,
\begin{equation*}
    E^2(u(T)) -E^2(u(0)) = i\int_{0}^{T} \Lambda_6(M_6; u(t)) \,\, dt,
\end{equation*}
where $M_6$ is given in (\ref{eq: M6}). We use the same notation as in Lemma \ref{lemma: E_2_vs_E_1}: we dyadically decompose $u = \sum_{j\geq 1} u_j$ where $u_j:=P_{N_j} u$ for dyadic $N_j \in 2^{\N_0}$ and without loss in generality assume that $\hat{u}$ is real and non-negative. Moreover, let $N_1^\star \geq \cdots \geq N_6^\star$ denote the decreasing ordering of  $(N_1, \dots, N_6)$ and define $u_j^\star:=P_{N_j^\star}u$. It suffices to show
\begin{equation*}
\sum_{N_1, \dots, N_6} \abs{\int_{0}^{T} \int_{\Gamma_6} M_6 \prod_{j=1}^6 \hat{u}_j(k_j) }\lesssim \lambda^{0+} N^{-\gamma(4\alpha-1)+}\prod_{j=1}^6\norm{Iu_j}_{X^{\alpha,\frac{1}{2}+}_T\left([0,T]\right)}.
\end{equation*}
Observe that $M_6$ vanishes unless $N_1^\star \gtrsim N$, so we restrict the above summation to $N_1^\star \sim N_2^\star \gtrsim N$. Moreover, $M_6$ is uniformly bounded on $\Gamma_6$ since $M_4$ is uniformly bounded on $\Gamma_4$ by Lemma \ref{lemma: M_4_bounded}. Thus, it suffices to show
\begin{equation*}
\int_{0}^{T} \int_{\Gamma_6} \prod_{j=1}^6 \hat{u}_j(k_j) \lesssim \left(N_1^\star \right)^{0-} \lambda^{0+} N^{-\gamma(4\alpha-1)+}\prod_{j=1}^6\norm{Iu_j}_{X^{\alpha,\frac{1}{2}+}_T\left([0,T]\right)}.
\end{equation*}
To simplify notation, we omit the time interval $[0,T]$ when referring to the time restricted modified Bourgain spaces. Accordingly, in what follows $X^{s,b}_T$ should be understood as $X^{s,b}_T\left([0,T]\right)$. \\
We will also use the following monotonicity property throughout the proof:
\begin{equation*}
    \frac{1}{m(k)\abs{k}^\beta} \lesssim N^{-\beta}, \text{ } \forall \beta>0, \text{ } \abs{k} \gtrsim N \text{ and } s \geq \alpha-\beta.
\end{equation*}
We now split the proof into two cases: (a) $ \lambda^{4\alpha-3} \geq N^{-(2\alpha-1)}$ and (b) $ \lambda^{4\alpha-3} < N^{-(2\alpha-1)}$. 
\begin{enumerate}
    \item[\textbf{Case (a)}] Here we have $ \lambda^{4\alpha-3} \geq N^{-(2\alpha-1)}$ and note that from the monotonicity property, $\frac{1}{m(N_j^\star)(N_j^\star)^{\beta-}} \lesssim \frac{1}{N^{\beta-}}$ for $j=1,2$. Therefore,
\begin{align*}
   \int_{0}^{T} \int_{\Gamma_6} \prod_{j=1}^6 \hat{u}_j^\star (k_j) & \lesssim  \frac{\left(N_1^\star\right)^{0-}}{N^{2\beta-}}\int_{0}^{T} \int_{\Gamma_6} \prod_{j=1}^2 \left(N_j^\star\right)^\beta m(N_j^\star)\hat{u}_j^\star \prod_{j=3}^6 \hat{u}_j^\star, \\
     &\lesssim \frac{\lambda^{0+} \left(N_1^\star\right)^{0-}}{N^{2\beta-}} \prod_{j=1}^2\norm{Iu_j^\star}_{X^{{\beta+\frac{1-\alpha}{3}+,\frac{1}{2}+}}_T} \prod_{j=3}^6\norm{u_j^\star}_{X^{\frac{1-\alpha}{3}+,\frac{1}{2}+}_T}, \\
    &\lesssim \frac{\lambda^{0+} \left(N_1^\star\right)^{0-}}{N^{\frac{2}{3}\left(4\alpha-1\right)-}} \norm{Iu}_{X_T^{\alpha,\frac{1}{2}+}}^6,
\end{align*}
where in the second line we used Plancherel's theorem, followed by Hölder's inequality and Lemma \ref{lemma:Linear X^0 Strichartz} in the third line. In the last line we select $\beta=\frac{4\alpha-1}{3}-$ and impose the condition $s>\frac{1-\alpha}{3}$. 
\item[\textbf{Case (b)}] Here we have $ \lambda^{4\alpha-3} < N^{-(2\alpha-1)}$ and we further distinguish two subcases: (i) $ N_3^\star \gtrsim N$ and (ii) $ N_3^\star \ll N$.\\
For subcase (i), we have $N_1^\star, N_2^\star , N_3^\star \gtrsim N$ and so, by the monotonicity property, we deduce that $\frac{1}{m(N_j^\star)(N_j^\star)^{\beta-}} \lesssim \frac{1}{N^{\beta-}}$ for $j=1,2,3$. An identical argument to case (a) yields 
\begin{equation*}
   \int_{0}^{T} \int_{\Gamma_6} \prod_{j=1}^6 \hat{u}_j^\star (k_j) \lesssim \frac{\lambda^{0+} \left(N_1^\star\right)^{0-}}{N^{\left(4\alpha-1\right)-}} \norm{Iu}_{X_T^{\alpha,\frac{1}{2}+}}^6.
\end{equation*}
For subcase (ii), note that $N_1^\star \sim N_2^\star \gg N_3^\star$, so we may apply the refined bilinear estimate from Lemma \ref{lemma: Bilinear X^0 Strichartz}. Moreover, since $N_3^\star \ll N$, we have that $m(N_3^\star)=m(N_4^\star)=m(N_5^\star)=m(N_5^\star)=1$. Therefore, 
\begin{align*}
\int_{0}^{T} \int_{\Gamma_6} \prod_{j=1}^6 \hat{u}_j^\star (k_j) & \lesssim  \frac{\left(N_1^\star\right)^{0-}}{N^{2\alpha-}}\int_{0}^{T} \int_{\Gamma_6} \prod_{j=1}^2 \left(N_j^\star\right)^\alpha m(N_j^\star)\hat{u}_j^\star \prod_{j=3}^6 m(N_j^\star)\hat{u}_j^\star\\
&\lesssim \frac{ \left(N_1^\star\right)^{0-}}{N^{2\alpha-}} \norm{Iu_1^\star Iu_3^\star}_{L_{t,x}^2} \norm{Iu_2^\star Iu_4^\star}_{L_{t,x}^2} \prod_{j=5}^6 \norm{Iu_j^\star}_{L^\infty_{t,x}} \\
     &\lesssim \frac{\left(N_1^\star\right)^{0-}}{N^{4\alpha-1-}} \prod_{j=1}^4\norm{Iu_j^\star}_{X_T^{\alpha,\frac{1}{2}+}} \prod_{j=5}^6\norm{Iu_j^\star}_{L^\infty_{t,x}} \\
    &\lesssim \frac{\left(N_1^\star\right)^{0-}}{N^{4\alpha-1-}} \left(N_5^\star N_6^\star \right)^{\frac{1}{2}-\alpha}\norm{Iu}_{X_T^{\alpha,\frac{1}{2}+}}^6,
\end{align*}
where the second line follows from Hölder's inequality, the third line follows from Lemma \ref{lemma: Bilinear X^0 Strichartz} and the last line follows from the Fourier support of $u_5^\star$ and $u_6^\star$. \qedhere
\end{enumerate}
\end{proof}
\end{lemma}
 We can now prove Theorem \ref{theorem:main_gwp_theorem_n=1} using Lemmata \ref{lemma: lwp_ILNS}, \ref{lemma: E_2_vs_E_1} and \ref{lemma: growth E2}. 
\begin{proof}[Proof of Theorem  \ref{theorem:main_gwp_theorem_n=1}] 
Suppose that $s\geq \frac{1-\alpha}{2}$ and let $u_0 \in H^s(\T)$. Let $u^\lambda$ denote the solution to (\ref{eq:fractional_main_lwp}) with initial condition $u^\lambda(x,0)=u^\lambda_0(x):= \lambda^{-\alpha} u_0(\frac{x}{\lambda})$. We stress that $u^\lambda_0 : \T_\lambda \rightarrow 
\C$ whereas $u_0: \T \rightarrow \C$ and hence all $L^p$ norms ($p \geq 1$) in the remainder of the proof are defined on the appropriate domain ($\T_\lambda$ for $u^\lambda_0$ and $\T$ for $u_0$).
For every $u_0 \in H^s$ and every $N \gg 1$, we select $\lambda = N^{\frac{\alpha-s}{s}} \norm{u_0}_{H^s}^{\frac{1}{s}}$ such that $E^1(u_0^\lambda) \lesssim_{\norm{u_0}_2} \lambda^{1-2\alpha}$ so that Lemma \ref{lemma: lwp_ILNS} can be applied. In particular note that
\begin{equation}
    \norm{(-\Delta)^\frac{\alpha}{2} Iu_0^\lambda}_{2}^2=\norm{m(k)\abs{k}^\alpha\hat{u}_0^\lambda}_{L^2((dk)_\lambda)}^2 \lesssim \frac{N^{2\alpha-2s}}{\lambda^{2\alpha+2s-1}} \norm{u_0}_{\dot{H}^s}^2 \sim \lambda^{1-2\alpha},
    \label{eq: energy_derivative_bound}
\end{equation}
and by the Gagliardo-Nirenberg inequality\cite{Nirenberg}, 
\begin{align}
    \norm{Iu_0^\lambda}_{4}^4 &\lesssim \norm{(-\Delta)^\frac{\alpha}{2}Iu_0^\lambda}_{2}^{\frac{1}{\alpha}} \norm{Iu_0^\lambda}_{2}^{4-\frac{1}{\alpha}} + \norm{Iu_0^\lambda}_{2}^{4} \notag \\ 
    &\lesssim \lambda^{2(1-2\alpha)}\left(\frac{N^{2\alpha-2s}}{\lambda^{2s}}\right)^\frac{1}{2\alpha} \norm{u_0}_{\dot{H}^s}^{\frac{1}{\alpha}}\norm{u_0}_{2}^{4-\frac{1}{\alpha}}+\lambda^{2(1-2\alpha)}\norm{u_0}_{2}^{4},  \notag\\
      &\lesssim \lambda^{2(1-2\alpha)} \norm{u_0}_2^4,
      \label{eq: energy_potential_bound}
\end{align}
where the penultimate estimate follows by $\norm{Iu_0^\lambda}_{2}^2 \leq \lambda^{1-2\alpha} \norm{u_0}_{2}^2$. Note that by the conservation of mass, $\norm{u_0}_2$ is constant so we can safely drop its dependence in the remainder of the estimates. We also drop the dependence of $\alpha$ and $s$ in all constants. \\
At time $t=0$, \eqref{eq: energy_derivative_bound} and \eqref{eq: energy_potential_bound} imply that $E^1(u_0^\lambda) \leq C \lambda^{1-2\alpha}$. We will next show that there exists $N_0=N_0(\norm{u_0}_{H^s})$ such that
\begin{equation}
    E^1(u^\lambda(nT_0)) \leq 2C \lambda^{1-2\alpha} \text{ for all } n \ll N^{\gamma(4\alpha-1)-}\lambda^{2(2\alpha-1)} \text{ and } N\geq N_0,
    \label{eq: continuity argument statement}
\end{equation}
where $T_0 \sim \lambda^{2(2\alpha-1)-}$ and $\gamma$ is defined in Lemma \ref{lemma: growth E2}. \\
To this end, define $\overline{n}=\min\{n\geq 0: E^1(u^\lambda(nT_0)) > 2 C \lambda^{1-2\alpha}\}$. We will argue by contradiction. Fix $n_0 = c N^{\gamma(4\alpha-1)-}\lambda^{2(2\alpha-1)}$ where $0<c<1$ is a small constant to be determined and assume that $\overline{n}<n_0$. Next, select $N_0$ large enough so that we can apply Lemma \ref{lemma: lwp_ILNS}. By Lemmata \ref{lemma: growth E2} and \ref{lemma: lwp_ILNS} followed by Lemma \ref{lemma: E_2_vs_E_1} we have that
\begin{align*}
    E^2(u^\lambda(\overline{n} T_0)) &\leq  E^2(u^\lambda(0)) + C_1n_0 N^{-\gamma(4\alpha-1)+}\lambda^{-3(2\alpha-1)}, \\
    &\leq E^1(u^\lambda(0))+ C_1n_0 N^{-\gamma(4\alpha-1)+}\lambda^{-3(2\alpha-1)}+ C_2N^{-\alpha}\lambda^{-2(2\alpha-1)}, \\
    &\leq \lambda^{1-2\alpha}\left(C+ C_1n_0 N^{-\gamma(4\alpha-1)+}\lambda^{-(2\alpha-1)}+ C_2N^{-\alpha}\lambda^{-(2\alpha-1)}\right).
\end{align*} 
Moreover, by Lemma \ref{lemma: E_2_vs_E_1} followed by Lemma \ref{lemma: lwp_ILNS}
\begin{align*}
    E^1(u^\lambda(\overline{n} T_0)) &\leq E^2(u^\lambda(\overline{n} T_0))+C_2N^{-\alpha} \lambda^{2(2\alpha-1)},\\
   &\leq \lambda^{1-2\alpha}\left(C+ C_1n_0 N^{-\gamma(4\alpha-1)+}\lambda^{-(2\alpha-1)}+ 2C_2N^{-\alpha}\lambda^{-(2\alpha-1)}\right).
\end{align*}
Therefore, by possibly increasing $N_0$ and by selecting $c\ll 1$, we have shown that $E^1(u^\lambda(\overline{n} T_0)) \leq 2 C \lambda^{1-2\alpha}$ which is a contradiction.  \\
By undoing the rescaling on \eqref{eq: continuity argument statement}, we obtain solutions to the original (fNLS) satisfying
\begin{equation*}
    \norm{u(t)}_{\dot{H}^s(\mathbb{T})}^2 \leq \lambda^{2\alpha-1+2s} \norm{u^\lambda(t\lambda^{2\alpha})}_{\dot{H}^s(\T_\lambda)}^2 \lesssim \lambda^{2s} \sim N^{2(\alpha-s)} \norm{u_0}_{\dot{H}^s(\T)}^2
\end{equation*}
 for $t\in [0,T]$, where $T \sim N^{\gamma(4\alpha-1)-}\lambda^{2(3\alpha-2)}  \sim  N^{\gamma(4\alpha-1)+\frac{2(\alpha-s)(3\alpha-2)}{s}-}\norm{u_0}_{H^s}^{\frac{2(3\alpha-2)}{s}}$. Therefore, we may deduce that solutions can be continued for time arbitrary intervals as long as,
 \begin{equation*}
     \lim_{N \rightarrow + \infty} N^{\gamma(4\alpha-1)-}\lambda^{2(3\alpha-2)} = + \infty.
 \end{equation*}
This is indeed true. If $\alpha \geq \frac{2}{3}$ this is clear, so it suffices to only consider $\alpha<\frac{2}{3}$. The case $\lambda^{4\alpha-3} \geq N^{-(2\alpha-1)}$ is trivially handled since
\begin{equation*}
    N^{\frac{2}{3}(4\alpha-1)-}\lambda^{2(3\alpha-2)} \geq \lambda^{\frac{2}{3}\frac{(3-2\alpha)(1-\alpha)}{(2\alpha-1)}-}.
\end{equation*}
Finally, the case $\lambda^{4\alpha-3} < N^{-(2\alpha-1)}$ follows from the fact that,
\begin{equation*}
    (4\alpha-1)s+2(\alpha-s)(3\alpha-2)>0 \iff s>\frac{2\alpha(2-3\alpha)}{3-2\alpha}. \qedhere
\end{equation*}
\end{proof}
\section{Proof of Theorem 
\ref{theorem:main_illposedness_theorem_n=1}} \label{sec:5}
 We introduce notation and state elementary lemmata used throughout this section. Define the frequency shift operator $M_n f(x):= e^{inx} f(x)$ for $n \in \Z$ and let $\psi(z):=\abs{z}^2z$ denote the cubic nonlinearity. Also recall the translation operator $T_b f(x):=f(x+b)$ defined in \cref{sec: I-operator}. We next record an approximate Galilean invariance identity for the fractional Schrödinger group.
\begin{lemma}
\label{lemma:galilean}
    Suppose $l, n \in \N$ with $n\geq l $ and let $f \in L^2(\T)$ satisfy $\supp \hat{f} \subseteq [-l,l]$. Then,
\begin{equation*}
        S(t) M_n f(x)= e^{in^{2\alpha}t}M_nT_{bt}f(x) +w(x,t),
    \end{equation*}
where $b:=2\alpha n^{2\alpha-1}$ and the error term, $w \in L^2(\T)$, satisfies 
\begin{equation*}
    \abs{\widehat{w}(k,t)} \leq \abs{t} l^2 n^{2\alpha-2}\abs{\widehat{f}(k)} \text{ for all } t\in \R, k \in \Z. 
\end{equation*}
\begin{proof}
    A straightforward computation yields
    \begin{equation*}
        S(t)e^{i(n+k)x} =  e^{i(k+n)x+i(k+n)^{2\alpha} t} =e^{i(k+n)x+in^{2\alpha}\left(1+2\alpha \frac{k}{n}+r_k\right)t},
    \end{equation*}
where the remainder satisfies $0 \leq r_k \leq \frac{k^2}{n^2}$. Thus,
\begin{equation*}
  S(t)e^{i(n+k)x}=e^{i n^{2\alpha} t} e^{inx} e^{ik(x+bt)}\left(1+\delta_{n,k}\right) ,
\end{equation*}
where $\delta_{n,k}:=e^{in^{2\alpha}r_kt}-1$.
\end{proof}
\end{lemma}
 Additionally, we establish an elementary $L^2$ bound for mixed products of functions assuming pointwise dominance of their Fourier coefficients.
\begin{lemma}
\label{lemma: convolution}
    Let $p,q \in \N$. Then for any $f,g \in L^2(\T)$ satisfying $\hat{f}(k) \geq 0$ and $\abs{\widehat{g}(k)} \leq \widehat{f}(k)$ for all $k \in \Z$, we have the following estimate
    \begin{equation*}
        \norm{f^p g^q}_2 \leq \norm{f^{p+q}}_2.
    \end{equation*}
\begin{proof}
    This lemma follows immediately from Parseval's theorem once we establish the following pointwise inequality $\left| \widehat{g^q}(k) \right|\leq \widehat{f^q}(k)$ for all $k \in \Z$ and $q \in \N$. We prove this by induction on $q$. The base case $q=1$ follows by assumption. For the inductive step assume that $\abs{\widehat{g^q}} \leq \widehat{f^q}.$ Then,
    \begin{equation*}
        \abs{\widehat{g^{q+1}}} \leq \abs{\widehat{g^{q}}}\ast \abs{\widehat{g}} \leq \widehat{f^q} \ast \widehat{f} = \widehat{f^{q+1}}. \qedhere  
    \end{equation*}
\end{proof}
\end{lemma}
 We now apply the preceding lemmata to prove the main ill-posedness result.
\begin{proof}[Proof of Theorem \ref{theorem:main_illposedness_theorem_n=1}]
     By \cite{Tao_disperive_book}, it suffices to prove that the map
    \begin{equation*}
        u_0 \mapsto \int_0^t S(t-t') \psi\left(S(t')u_0\right) dt',
    \end{equation*}
    is not bounded from $H^s(\T)$ to $C^0([0,T] \rightarrow H^s(\T))$ for $s<\frac{1-\alpha}{2}$ and any $T \ll 1$. To this end, define $u_0:=M_n f$ where
\begin{equation*}
    f(x):=n^{\frac{\alpha-1}{2}-s} \sum_{k=0}^{l_n} e^{ikx},
\end{equation*}
$l_n:=\lfloor n^{1-\alpha} \rfloor$ and $n \in \N$. Note that $\norm{u_0}_{H^s} \sim 1$. We will next show that for all $0 \leq t' \leq t \leq T \ll 1$,
\begin{equation}
    \norm{S(t-t')\psi\left(S(t')M_n f\right)-e^{in^{2\alpha} t} M_nT_{bt}\psi(f)}_2 \ll \norm{\psi(f)}_2.
    \label{eq:main_differene}
\end{equation}
First note that by Lemma \ref{lemma:galilean} followed by Lemma \ref{lemma: convolution}, we have the following estimate
\begin{align*}
   \norm{\psi\left(S(t')M_n f\right)-e^{in^{2\alpha} t'}M_n T_{bt'}\psi(f)}_2 &\lesssim \norm{f^2w}_2+\norm{fw^2}_2+\norm{w^3}_2,  \\
    &\ll \norm{\psi(f)}_2,
\end{align*}
and since $S(t)$ is unitary on $L^2(\T)$,
\begin{equation}
   \norm{S(t-t')\left(\psi\left(S(t')M_n f\right)-e^{in^{2\alpha} t'}M_n T_{bt'}\psi(f)\right)}_2 
    \ll \norm{\psi(f)}_2.
    \label{eq:part1}
\end{equation}
Another application of Lemma \ref{lemma:galilean} yields
\begin{equation}
   \norm{e^{in^{2\alpha} t'}S(t-t')M_n T_{bt'}\psi(f)-e^{in^{2\alpha} t}M_n T_{bt}\psi(f)}_2 
    \ll \norm{\psi(f)}_2.
    \label{eq:part2}
\end{equation}
Combining \eqref{eq:part1}, \eqref{eq:part2} and the triangle inequality yields \eqref{eq:main_differene}. Thus, we have shown that for $0 \leq t \leq T \ll 1$,
\begin{equation*}
    \norm{\int_0^t S(t-t') \psi\left(S(t')u_0\right) dt'}_{H^s} \gtrsim t n^s \norm{\psi(f)}_2 \sim t n^{1-\alpha-2s}. \qedhere
\end{equation*}
\end{proof}
\section{Acknowledgment}
The authors are grateful to Professor Larry Guth and Professor Gigliola Staffilani for their valuable feedback during the preparation of this manuscript. The authors also appreciate the careful proofreading and helpful suggestions of Paige Bright and Jiahui Yu.
\bibliography{references}
\bibliographystyle{amsplain2.bst}

\bigskip

\end{document}